\pdfoutput=1

\documentclass[a4paper,10pt,reqno]{amsart}
\usepackage{initial-data-kerr-stability}

\numberwithin{equation}{section}

\pgfdeclaredecoration{complete sines}{initial}
{
  \state{initial}[
  width=+0pt,
  next state=sine,
  persistent precomputation={\pgfmathsetmacro\matchinglength{
      \pgfdecoratedinputsegmentlength / int(\pgfdecoratedinputsegmentlength/\pgfdecorationsegmentlength)}
    \setlength{\pgfdecorationsegmentlength}{\matchinglength pt}
  }] {}
  \state{sine}[width=\pgfdecorationsegmentlength]{
    \pgfpathsine{\pgfpoint{0.25\pgfdecorationsegmentlength}{0.5\pgfdecorationsegmentamplitude}}
    \pgfpathcosine{\pgfpoint{0.25\pgfdecorationsegmentlength}{-0.5\pgfdecorationsegmentamplitude}}
    \pgfpathsine{\pgfpoint{0.25\pgfdecorationsegmentlength}{-0.5\pgfdecorationsegmentamplitude}}
    \pgfpathcosine{\pgfpoint{0.25\pgfdecorationsegmentlength}{0.5\pgfdecorationsegmentamplitude}}
  }
  \state{final}{}
}

\makeatletter
\tikzset{
  hatch distance/.store in=\hatchdistance,
  hatch distance=5pt,
  hatch thickness/.store in=\hatchthickness,
  hatch thickness=5pt
}
\pgfdeclarepatternformonly[\hatchdistance,\hatchthickness]{north east hatch}
{\pgfqpoint{-1pt}{-1pt}}
{\pgfqpoint{\hatchdistance}{\hatchdistance}}
{\pgfpoint{\hatchdistance-1pt}{\hatchdistance-1pt}}%
{
  \pgfsetcolor{\tikz@pattern@color}
  \pgfsetlinewidth{\hatchthickness}
  \pgfpathmoveto{\pgfqpoint{0pt}{0pt}}
  \pgfpathlineto{\pgfqpoint{\hatchdistance}{\hatchdistance}}
  \pgfusepath{stroke}
}
\makeatother
\newcommand{\EventHorizonFuture}{\mathcal{H}^+}

\newcommand{\EventHorizonPast}{\mathcal{H}^-}

\newcommand{\timelikeInf}{i^+}
\newcommand{\timelikeNegInf}{i^-}

 \author{Allen Juntao Fang}
 \address{Universit\"at M\"unster, M\"unster,
   Deutschland
   (\href{mailto:allen.juntao.fang@uni-muenster.de}{allen.juntao.fang@uni-muenster.de})}
 \author{J\'er\'emie Szeftel}
 \address{CNRS \& Laboratoire Jacques-Louis Lions, Sorbonne Universit\'e, Paris, France
   (\href{mailto:jeremie-szeftel@sorbonne-universite.fr}{jeremie-szeftel@sorbonne-universite.fr})}
 \author{Arthur Touati}
 \address{CNRS \& Institut Mathématique de Bordeaux, Université de Bordeaux, France
   (\href{mailto:arthur.touati@math.u-bordeaux.fr}{arthur.touati@math.u-bordeaux.fr})}

\title{Spacelike initial data for black hole stability}

\begin{document}

\begin{abstract}
  We construct initial data suitable for the Kerr stability conjecture, that is,
  solutions to the constraint equations on a spacelike hypersurface
  with boundary entering the black hole horizon that are arbitrarily
  decaying perturbations of a Kerr initial data set. This results
  from a more general perturbative construction on any 
  asymptotically flat initial data set with $r^{-1}$ fall-off and the topology of
  $\RRR^3\setminus\{r<1\}$ enjoying some analyticity near and at the
  boundary. In particular, we design a suitable mixed boundary
  condition for the elliptic operator of the conformal method in order
  to exclude the Killing initial data sets (KIDS).
\end{abstract}

\maketitle

\section{Introduction}

The Kerr family of
stationary black holes is generally believed to
describe the black holes observed in astrophysics. An important question regarding
whether or not this belief is substantiated is the stability property
of this family. In this article, we construct initial data for the Kerr stability problem.

\subsection{The initial value problem in general relativity}

The Einstein vacuum equations (EVE) are the defining equations of
Einstein's theory of general relativity under the assumption of vacuum
and are given by
\begin{equation}
  \label{EVE}
  \Ric(\g) = 0,
\end{equation}
where $\g$ is a Lorentzian metric defined on a 4-dimensional manifold
$\MM$, and $\Ric$ denotes the Ricci tensor. As first shown in the seminal papers
\cite{ChoquetBruhat1952,ChoquetBruhat1969a},
there is a well-defined Cauchy problem for \eqref{EVE}, where
one solves for a maximal globally hyperbolic development $(\MM, \g)$
of a given initial data triplet $(\Sigma, g, k)$ consisting of a
3-dimensional Riemannian manifold $(\Sigma, g)$ and a symmetric two
tensor $k$ satisfying the constraint
equations. The constraint equations themselves are given by
\begin{equation}\label{constraint equations}
  \begin{aligned}
    R(g)  + (\tr_g k)^2 - |k|^2_g & = 0,
    \\ \div_g k - \mathrm{d}\tr_g k & = 0,
  \end{aligned}
\end{equation}
where $R(g)$ is the scalar curvature of $g$. 

\subsection{The Kerr stability conjecture}\label{section conjecture}

Studying the stability properties of black hole solutions to EVE can
be formulated in the context of the above Cauchy problem for EVE. We provide below a rough statement of the Kerr stability conjecture, and
refer the reader for example to the introduction of \cite{Giorgi2022,Klainerman2023} for a more detailed discussion.
  
\begin{conjecture}[Kerr stability conjecture]
  \label{conjecture}
  If $(\Sigma,g,k)$ is an initial data triplet such that $(g,k)$ is an
  appropriately small perturbation of $(g_{m,a}, k_{m,a})$ the initial
  data of a subextremal Kerr black hole $\g_{m,a}$ (i.e. $|a|<m$), then the evolution of
  $(\Sigma, g, k)$ under \eqref{EVE}, $(\MM, \g)$, has a domain of outer communication that converges in the
  appropriate sense to a nearby member of the Kerr family
  $(\MM_{m_f, a_f},\g_{m_f, a_f})$.
\end{conjecture}

In the context of black hole stability, there are generally three
regions of interest in a black hole spacetime, separated by null cones
arising from the initial data, as indicated by the three regions $I$,
$II$, and $III$ in \Cref{fig:penrose}. We discuss briefly the
stability results in each of these regions.

\begin{figure}[H]
  \begin{tikzpicture}[scale=0.7,every node/.style={scale=1.3}]


    \def \s{5} 

    \coordinate (tInf) at (0,\s); 
    \coordinate (EventZero) at (-\s,0); 
    \coordinate (CosmoZero) at (\s,0); 
    \coordinate (tNegInf) at (0,-\s);
    \coordinate (EventSigmaZero) at (-0.8*\s,0.2*\s);

    \draw[name path=EventFuture] (tInf) --
    node[pos=0.5,left]{$\EventHorizonFuture$} (EventZero) ;  
    \draw[name path=CosmoFuture,dotted] (tInf) --
    node[pos=0.5,right]{$\;\mathscr{I}^{+}$} (CosmoZero) ; 
    \draw[name path=EventPast] (tNegInf) --
    node[pos=0.5,left]{$\EventHorizonPast$} (EventZero) ; 
    \draw[name path=CosmoPast,dotted] (tNegInf) --
    node[pos=0.5,right]{$\mathscr{I}^{-}$} (CosmoZero);

    \tikzfillbetween[of=EventFuture and EventPast]{black, opacity=0.1};
    \tikzfillbetween[of=CosmoFuture and CosmoPast]{black, opacity=0.1};

    \path[name path=SigmaZero,black,thick,out=-30,in=-175,shorten >= 0, shorten <= -10]
    (EventSigmaZero) edge coordinate[pos=0.4](SplitPoint) node[pos=0.5,below]{$\Sigma$} (CosmoZero);
    \path[name path=SplitPointFutureIngoing] (SplitPoint) -- ($(SplitPoint)+(-0.7*\s,0.7*\s)$);
    \path[name intersections={of=EventFuture and SplitPointFutureIngoing, by=EventFutureSplitPoint}];
    \path[black, thick, dashed] (EventFutureSplitPoint) edge (SplitPoint);
    \path[name path=SplitPointFutureOutgoing] (SplitPoint) -- ($(SplitPoint)+(0.7*\s,0.7*\s)$);
    \path[name intersections={of=CosmoFuture and SplitPointFutureOutgoing, by=CosmoFutureSplitPoint}];
    \path[black, thick, dashed] (CosmoFutureSplitPoint) edge (SplitPoint);

    \node[scale=0.3,fill=black,draw,circle]at(SplitPoint){};

    \node at (barycentric cs:EventSigmaZero=1,EventFutureSplitPoint=1,SplitPoint=1){$I$};
    \node at (barycentric cs:tInf=1,EventFutureSplitPoint=1,SplitPoint=1,CosmoFutureSplitPoint=1){$III$};
    \node at (barycentric cs:CosmoZero=1,CosmoFutureSplitPoint=1,SplitPoint=1){$II$};        
    
    \node[scale=0.5,fill=white,draw,circle,label=above:$\timelikeInf$]at(tInf){};
    \node[scale=0.5,fill=white,draw,circle,label=below:$\timelikeNegInf$]at(tNegInf){};
    \node[scale=0.5,fill=white,draw,circle,label=below:$i^0$]at(CosmoZero){};

  \end{tikzpicture}

  \centering
  \caption{Penrose diagram of the future of initial perturbations of Kerr.}
  \label{fig:penrose}
\end{figure}
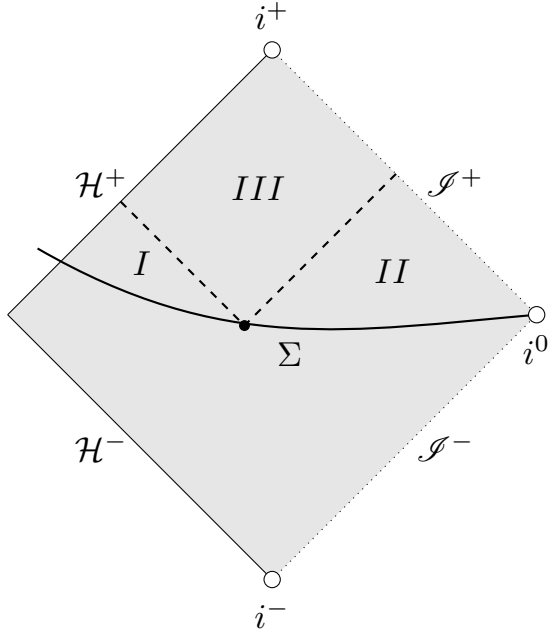
Region $I$ is a compact region of spacetime, and thus, is a local
existence regime. In region $II$, the full subextremal Kerr family
was shown to be stable by
\cite{Caciotta2010,Shen2024}. In region $III$, stability has been shown\footnote{See also \cite{Klainerman2020,Dafermos2021} for the stability of Schwarzschild (i.e. $a=0$) for restrictive initial data.} for slowly-rotating Kerr black holes (i.e. $|a|\ll m$) in \cite{Giorgi2022,Klainerman2022,Klainerman2022a,Klainerman2023,Shen2023c}, using initialization\footnote{This initialization takes place on the null cones of Figure \ref{fig:penrose} where the constraint equations take the form of nonlinear ODEs, as opposed to \eqref{constraint equations} which are non-linear elliptic PDEs.} from regions $I$ and $II$, thus proving Conjecture \ref{conjecture} in the slowly-rotating case.

\subsection{First version of the main result}

Our goal is to construct $(g, k)$ such that $(\Sigma, g, k)$ is
an initial data triplet satisfying the constraint equations \eqref{constraint equations} and is an
arbitrarily decaying perturbation of the initial data for a given Kerr black hole in order to satisfy the requirements in
\cite{Caciotta2010,Shen2024}. We remark that, by simply taking $a=0$ or small compared to the mass, this will also induce
data that suffices for
\cite{Dafermos2021,Klainerman2020,Klainerman2023}. We
give here a rough version of our main result, see Theorem \ref{coro
  kerr} for the precise statement.

\begin{theorem}[Rough statement]
  \label{theo rough}
  Let $(\Sigma, g_{m,a}, k_{m,a})$ be the spacelike initial
  data set induced by the Kerr metric $\g_{m,a}$, and let\footnote{We denote by $\mathbb{N}$ the set of nonnegative integers and by $\mathbb{N}^*$ the set of positive integers.}
  $q\in\mathbb{N}^*$, $0<\de<1$ and $\e>0$ small enough. There
  exists a codimension $4q^2$ family of perturbations $\pth{g_{m,a} + \tilde{g}, k_{m,a}+\tilde{k}}$
  of $(g_{m,a}, k_{m,a})$ such that
  \begin{equation}\label{decay theorem}
    \abs*{\tilde{g}} \sim \varepsilon (1+r)^{-q-\delta}, \quad \abs*{\tilde{k}} \sim \varepsilon (1+r)^{-q-\delta-1},
  \end{equation}
  and $\pth{\Sigma, g_{m,a} + \tilde{g}, k_{m,a}+\tilde{k}}$
  satisfies the constraint equations \eqref{constraint equations}.
\end{theorem}

Theorem \ref{theo rough} provides initial data for the works mentioned in Section \ref{section conjecture}:
\begin{itemize}
    \item In region $II$, \cite{Caciotta2010} requires $q\geq 3$ and \cite{Shen2024} requires $q\geq 1$, with $0<\de<1$ in both cases.
    \item In region $III$, \cite{Klainerman2020,Klainerman2023} require $q=1$ and \cite{Dafermos2021} requires $q=2$, with, in both cases, $\de=\half + \de'$ and $0<\de'\ll 1$.
\end{itemize}

\subsection{Construction of initial data sets}

There is a rich literature on the constraint equations \eqref{constraint equations} and we refer the reader for example to \cite{Carlotto2021} for a detailed review. The main two approaches to produce solutions to these equations are the conformal method and the gluing method. First introduced in \cite{Lichnerowicz1944}, the conformal method proposes a conformal ansatz for $(g,k)$ and transforms the quasi-linear underdetermined system \eqref{constraint equations} into a semi-linear elliptic system for a scalar function and a vector field. This approach has been successful in constructing many interesting solutions to the constraint equations, including those with large mean curvature or with apparent horizons \cite{Avalos2022,Dilts2014,Holst2009,Holst2015,Holst2018,Maxwell2005,Maxwell2009}. However these results fall short of constructing perturbations of a given black hole satisfying \eqref{decay theorem}. This is due to the fact that these constructions are typically done within the range of invertibility of the relevant elliptic operators resulting in perturbations that decay like $r^{-\delta}$ for $0<\de<1$, falling short\footnote{For instance, applying Theorem 1.17 from \cite{Bartnik1986} could improve the decay of the solutions but only modulo harmonic polynomials.} of the decay rates assumed in the aforementioned stability results mentioned right after Theorem \ref{theo rough}.

The gluing method, first introduced in \cite{Corvino2000} and later extended in \cite{Corvino2006,Chrusciel2003}, consists of gluing any asymptotically flat solution of \eqref{constraint equations} to an exact Kerr initial data set outside a far away annulus. However, critically, these constructions, viewed as a perturbation of the background black hole metric, are compactly supported, whereas we aim to construct non-compactly supported initial data.

\begin{remark}
Gluing results use the Kerr initial data, viewed as a 10-parameter family, to saturate the 10-dimensional space of linear obstructions coming from the symmetries of Minkowski. Such families are called admissible families in \cite{Corvino2006} (see also \cite{Corvino2020}) and reference families in \cite{Chrusciel2003}. Theorem \ref{theo rough} exhibits infinitely many such admissible or reference families with the same asymptotics as Kerr at main order. As a result, one should be able to glue asymptotically flat initial data to arbitrarily decaying perturbations of Kerr. 
\end{remark}

\begin{remark}
In Theorem 8.9 of \cite{Chrusciel2003}, initial data coinciding to arbitrary order with Simon-Beig approximate stationary vacuum spacetimes (see \cite{Simon1983}) are constructed. This gives particular examples of initial data sets close to a given Kerr with arbitrary fall-off. Given that they are parametrized by the formal stationary solutions of \cite{Simon1983}, they form a very thin subset of the solutions constructed in this paper. Indeed, consider as an analogy the scalar wave equation in $\RRR^{1+3}$ with data $(\psi_0,\psi_1)$: in that context, the equivalent of our solutions is the whole space $H^2_{-q-\de}(\RRR^3)\times H^1_{-q-\de-1}(\RRR^3)$ while the equivalent of the solutions constructed in Theorem 8.9 of \cite{Chrusciel2003} is the subspace $\left\{\text{$\De\psi_0=0$ outside of a ball}\right\}\times\{\psi_1=0\}$.
\end{remark}

The authors have previously constructed and parametrized initial data perturbations of Minkowski space with arbitrary decay \cite{Fang2024}. While the overall methodology is similar, there are two key differences between the current paper and \cite{Fang2024}. We first recall that in \cite{Fang2024} the Killing initial data sets (KIDS) presented a fundamental, linear obstacle to the construction. Since black hole spacetimes have fewer Killing vectorfields than Minkowski, there are fewer KIDS to deal with in the present paper than in \cite{Fang2024}. In addition, motivated by the stability of Minkowski, initial data were constructed in \cite{Fang2024} for a manifold without boundary. On the other hand, in the present paper, initial data are constructed for a manifold with boundary, reflecting the presence of an event horizon.  The presence of the boundary adds flexibility, as the elliptic problem to be solved does not itself impose boundary conditions. This will play a crucial role in the ensuing proof, where we take advantage of this freedom.

\subsection{Sketch of proof}

We now give a brief sketch of the proof of Theorem \ref{theo rough}. Rather than considering initial data $(\Sigma,g, k )$ it will be convenient to instead consider initial data triplets $(\Sigma, g, \pi)$, where $\pi=k-(\tr_g k)g$ will be referred to as the \emph{reduced second fundamental form} of $\Sigma$.

Using the conformal method, we consider as an ansatz $(g,\pi)$ of the form
\begin{align} \label{ansatz rough}
(g,\pi)= \pth{ (1+\wc{u})^4 \pth{ g_{m,a} + \wc{g}}, (1+\wc{u})^2\pth{\pi_{m,a} + \wc{\pi} + L_{g_{m,a}}\wc{X}} },
\end{align}
where $L_{g_{m,a}}$ is a modified Lie derivative (see
\eqref{eq:modified-Lie-derivative:def}), and where $(\wc{g},\wc{\pi})$ are the "seed" of
the construction and enjoy smallness and decay satisfying
\begin{align} \label{seed assumption}
\sup_{k=0,1} \pth{ r^k | \nab^k \wc{g}| + r^{k+1}|\nab^k \wc{\pi}| } \leq \frac{\e}{(1+r)^{q+\de}}.
\end{align}

\begin{remark}
The conformal ansatz \eqref{ansatz rough} does not correspond to the most standard conformal formulation of the constraint, which usually makes use of the conformal Killing operator to distinguish between the tracefree and trace part of the second fundamental form. Here, our goal is to study the perturbative regime for the constraint equations, a setting where it suffices to consider the less precise ansatz of \cite{Corvino2006}.
\end{remark}

The constraint equations \eqref{constraint equations} for $(g,\pi)$ then become the following system of PDEs on $\Si$
\begin{align} \label{system rough}
P\pth{\wc{u},\wc{X}} = D\Phi[g_{m,a},\pi_{m,a}](\wc{g},\wc{\pi}) + \text{non-linear terms},
\end{align}
where $P$ is an elliptic operator depending only on $(g_{m,a},\pi_{m,a})$, and $D\Phi[g_{m,a},\pi_{m,a}]$ denotes the linearization of the constraint operator at $(g_{m,a},\pi_{m,a})$. \Cref{theo rough} then asks whether there exists a solution $\left( \wc{u},\wc{X} \right)$ to \eqref{system rough} satisfying 
\begin{align} \label{decay u X}
|\wc{u}| + r\left|\nab \wc{X}\right| \lesssim \frac{\e}{(1+r)^{q+\de}}.
\end{align}
We remark that applying the conformal method does not impose boundary conditions on $\dr\Si$ for the solutions $\pth{\wc{u}, \wc{X}}$ to \eqref{system rough}, which will play a critical role as mentioned above.

To address this question, we study the behavior of the elliptic operator $P$ on weighted Sobolev spaces\footnote{Recall that $u\in H^2_\eta(\Sigma)$ implies $|u|\lesssim r^\eta$ at infinity, see Definition \ref{def:Hkq}.}  $H^2_{-q-\de}(\Si)$.  Using qualitative properties of elliptic operators to understand the decay of their solutions is a classical method, and we illustrate the basic idea behind with toy model of the scalar Laplacian on Euclidean space before addressing the full problem at hand. The mapping properties of the Laplacian
\begin{equation}\label{eq:Laplacian-weighted-Sobolev-spaces}
\Laplacian: \HkqScalar{2}{\delta}(\Real^3)\longrightarrow \HkqScalar{0}{\delta-2}(\Real^3)
\end{equation}
are well studied \cite{Bartnik1986,McOwen1979}, and in particular, it is known that the Laplacian as defined in \eqref{eq:Laplacian-weighted-Sobolev-spaces} is a Fredholm operator, if $\de\notin\mathbb{Z}$. This implies, via the Fredholm alternative, that provided $h$ is orthogonal to the kernel of $\Laplacian^{*}$, then $\Laplacian \varphi =h$ is solvable. Thus, given some $h\in \HkqScalar{0}{\delta-2}(\Real^3)$, finding a perturbation $\tilde{h}$ such that $\Laplacian \varphi  =  h+\tilde{h}$ is solvable in $\HkqScalar{2}{\delta}(\Real^3)$ reduces to solving for $\tilde{h}$ such that
\begin{equation}\label{produit scalaire rough}
\ps{h+\tilde{h}, w}_{L^2(\Real^3)} = 0, \quad \text{for all } w\in \ker (\Laplacian^{*}).
\end{equation}
In this case of the Laplacian, the kernel of $\Laplacian^{*}$ can be explicitly described and it consists just of harmonic polynomials. Thus, a suitable perturbation $\tilde{h}$ can be explicitly solved for. In particular, this toy model emphasizes the fact that a full characterization of the cokernel of the elliptic operator under consideration is crucial.

We now return to giving a sketch of the proof in our setting. It can be verified that $P$ is also a Fredholm operator acting between weighted Sobolev spaces. However, characterizing the kernel of $P^{*}$ is more delicate. We characterize the kernel of $P^{*}$ as perturbations of the harmonic polynomials. The critical tool will be a good choice of mixed boundary condition on $\dr\Si$ of the form
\begin{align}\label{boundary condition intro}
(B_\nu + F)\pth{\wc{u},\wc{X}} = 0,
\end{align}
such that the operator $(P,B_\nu + F)$ acting on $H^2_{-\de}(\Si)$ for $0<\de<1$ will be injective, and where $B_\nu$ is a Neumann type boundary operator defined on $\dr\Si$. We will also require that $F$ vanishes identically on an open subset of $\partial\Sigma$ and is the identity on another open subset. This condition on $F$ is not necessary to characterize the kernel of $P^{*}$ but will play a key role in eliminating certain obstructions in what follows. 

\begin{remark}
As opposed to the boundary condition considered for instance in \cite{Maxwell2005}, our boundary condition \eqref{boundary condition intro} has no geometric interpretation. In addition to ensuring solvability of our elliptic problem with appropriately decaying solutions, it will also allow us to exclude the KIDS, as we will see below.
\end{remark}

To ensure the orthogonality conditions necessary for $P$ to be invertible, we will modify the ansatz
\eqref{ansatz rough} to include compactly supported symmetric
2-tensors $(\breve{g},\breve{\pi})$ such that the elliptic system
reads
\begin{align}\label{system rough 2}
  P\pth{\wc{u},\wc{X}} = D\Phi[g_{m,a},\pi_{m,a}]\pth{(\wc{g},\wc{\pi}) + (\breve{g},\breve{\pi}) } + \text{non-linear terms}.
\end{align}
The correctors $(\breve{g},\breve{\pi})$ will then be chosen in a $4q^2$ dimensional space such as to ensure the following analogue of \eqref{produit scalaire rough}:
\begin{align}\label{condition breve}
  \ps{(\breve{g},\breve{\pi}),D\Phi[g_{m,a},\pi_{m,a}]^*(\mathbb{W})}= - \ps{D\Phi[g_{m,a},\pi_{m,a}](\wc{g},\wc{\pi}),\mathbb{W}} + \text{non-linear terms}
\end{align}
for all $\mathbb{W}$ in the cokernel of $(P,B_\nu + F)$. The
linearization of the system in \eqref{condition breve} is solvable if
$D\Phi[g_{m,a},\pi_{m,a}]^*$ acts injectively on the cokernel of
$(P,B_\nu + F)$ restricted to the support of
$(\breve{g},\breve{\pi})$. Now, since the Kerr initial data
$(g_{m,a},\pi_{m,a})$ can be proved to be analytic near and at the
boundary, elements of $\mathrm{ker}\pth{D\Phi[g_{m,a},\pi_{m,a}]^*}$ can also be proved
to be analytic near and at the boundary (see Lemma \ref{lem regularity kids} for the precise statement). Therefore, they cannot
satisfy the boundary condition \eqref{boundary condition intro} since
we assumed that $F$ vanishes in an open set of $\partial\Sigma$ and is the identity on another one. This shows how $D\Phi[g_{m,a},\pi_{m,a}]^*$ acts
injectively on the cokernel of $(P,B_\nu + F)$ (see Lemma \ref{lem KIDS} for the details).

The above procedure illustrates how a careful choice of $F$ allows us to simultaneously
describe the mapping properties of $P$ and also to ensure that we can
appropriately solve for the correctors $(\breve{g}, \breve{\pi})$ by
eliminating the KIDS, i.e. the elements of $\mathrm{ker}(D\Phi[g_{m,a},\pi_{m,a}]^*)$, which are well-known linear obstructions to solving
\eqref{condition breve}. A straightforward fixed point argument then concludes the
resolution of the non-linear system \eqref{system rough 2}.

\begin{remark}\label{remark more general}
In fact, the only properties of the Kerr initial data set we use are its analyticity near and at the boundary, and the fact that it solves \eqref{constraint equations}. For this reason, we prove a more general statement on asymptotically flat solutions of \eqref{constraint equations} in the sense of Definition \ref{def AF} which are analytic near and at the boundary, see Theorem \ref{theo general}, from which the result for Kerr is then deduced, see Theorem \ref{coro kerr}.
\end{remark}

\subsection{Outline of the article}

We conclude this introduction by giving an outline of the article:
\begin{itemize}
\item In Section \ref{section preliminaries}, we start with geometric and analytic preliminaries.
\item In Section \ref{section main result}, we give a precise
  statement of our main result Theorem \ref{theo general}, from which we deduce Theorem \ref{theo rough}.
\item In Section \ref{section elliptic}, we study the properties of the elliptic
  operator $P$ at the heart of our construction.
\item Finally, in Section \ref{section proof main result} we prove
  Theorem \ref{theo general}.
\end{itemize}

\subsection*{Acknowledgments}

The first author acknowledges support from NSF award DMS-2303241
and support through Germany’s Excellence Strategy EXC 2044 390685587,
Mathematics Münster: Dynamics--Geometry--Structure. The second author is supported by the ERC grant ERC-2023 AdG 101141855 BlaHSt.

\section{Preliminaries}
\label{section preliminaries}

\subsection{Notations and geometry}

Throughout the paper, we will use Einstein summation notation, with roman
indices referring to spatial indices $\{1,2,3\}$. 

\begin{definition}
  \label{def AF}
  A triplet $(\Si,\bar{g},\bar{\pi})$, where $\Sigma$ is a 3-manifold
  with boundary, $\bar{g}$ is a Riemannian metric on $\Si$ and
  $\bar{\pi}$ is a symmetric 2-tensor is called \emph{asymptotically flat}
  if:
  \begin{itemize}
  \item[(i)] $\Si$ is smooth, diffeomorphic to $\RRR^3\setminus B$
    where $B$ is an open ball in $\RRR^3$ and there exists a global
    coordinate system $(x^1,x^2,x^3)$ on $\Si$ such that
    $\dr\Si=\{r=r_0\}$ for some $r_0>0$ where
    $r\vcentcolon=\sqrt{(x^1)^2+(x^2)^2 + (x^3)^2}$,
  \item[(ii)] $(\bar{g},\bar{\pi})$ is smooth and asymptotically
    flat, i.e. there exists constants $C_{k}$ such that for every
    $k\in\mathbb{N}$ we have
    \begin{align*}
      \sup_\Si \pth{ r^{k+1}\left| \nab^k (\bar{g}_{ij} - \de_{ij})\right| + r^{k+2}\left| \nab^k \bar{\pi}_{ij} \right| } \leq C_k,
    \end{align*}
    where $\nab=\pth{\dr_{x^1},\dr_{x^2},\dr_{x^3}}$.
  \end{itemize}
  We will also refer to $\Sigma$ satisfying the above conditions as an
  asymptotically flat manifold with boundary. 
\end{definition}

\begin{remark}
Note that more general definitions of asymptotic flatness can be found in the literature. For instance, \cite{Bartnik1986} considers metric such that $g_{ij}-\de_{ij}=\GO{r^{-\tau}}$ for $\tau>0$. Moreover, our choice of topology, i.e. $\RRR^3\setminus B$, is directly motivated by the Kerr stability conjecture's setting.
\end{remark}

\subsection{Function spaces}

On $(\Sigma,\bar{g},\bar{\pi})$ an asymptotically flat manifold with
boundary, as defined in \Cref{def AF}, we use the standard $L^p(\Si)$
spaces for $1\leq p \leq \infty$, where the integration is defined
with respect to the volume form of $\bar{g}$. Moreover, the boundary
$\dr\Si$ is a compact 2-manifold and we will use the standard
$H^s(\dr\Si)$ spaces for $s\in\RRR$. We now define function spaces
with weights at infinity.

\begin{definition}
  \label{def:Hkq}
  Let $(\Sigma,\bar{g},\bar{\pi})$ be an asymptotically flat manifold with boundary as
  defined in \Cref{def AF} and let $k\in \mathbb{N}$ and $\de\in\RRR$.  
  \begin{enumerate}[label=(\roman*)]
  \item We define the \emph{weighted Sobolev spaces} $H^k_\de(\Si)$ to be the set of functions in $L^2_{loc}(\Si)$ for which the following norm is finite
    \begin{equation*}
      \norm*{u}_{H^{k}_{\delta}(\Sigma)} \vcentcolon = \sum_{\abs*{\alpha}\le k}\norm*{(1+r)^{-\delta-\frac{3}{2}+\abs*{\alpha}}\nabla^{\alpha}u}_{L^2(\Sigma)}.
    \end{equation*}
    We denote $L^2_\de(\Si)=H^0_\de(\Si)$.
  \item We define the \emph{weighted Hölder spaces} $C^k_\de(\Si)$ to be the set of functions in $W^{k,\infty}_{loc}(\Si)$ for which the following norm is finite
    \begin{equation*}
      \norm*{u}_{C^{k}_{\delta}(\Sigma)} \vcentcolon = \sum_{\abs*{\alpha}\le k}\norm*{(1+r)^{-\de+\abs*{\alpha}}\nabla^{\alpha}u}_{L^\infty(\Sigma)}.
    \end{equation*}
  \item We extend these definitions to tensors of any type by summing
    over components in the global coordinates system $(x^1,x^2,x^3)$
    of \Cref{def AF}.
  \end{enumerate}
\end{definition}

In the following lemma, we gather standard properties of these
spaces. Proofs of these properties can be found in Appendix I of
\cite{ChoquetBruhat2009} or in \cite{ChoquetBruhat1981,Bartnik1986,Maxwell2006} (beware that we use Bartnik's convention for weighted Sobolev spaces which differs from Choquet-Bruhat's).

\begin{lemma}\label{lem embedding}
  Let $k,k_1,k_2\in\mathbb{N}$ and $\de,\de_1,\de_2\in\RRR$.
  \begin{itemize}
  \item[(i)] If $k\geq 1$ and $\de'<\de$ then the inclusion
    $H^k_{\de'}(\Si)\subset H^{k-1}_{\de}(\Si) $ is compact.
  \item[(ii)] If $k\geq 2$ then $H^k_\de(\Si)\subset C^0_\de(\Si)$ and
    this inclusion is continuous.
  \item[(iii)] If $k\leq \min(k_1,k_2)$, $k<k_1+k_2-\frac{3}{2}$ and
    $\de>\de_1+\de_2$ then multiplication is a continuous map from
    $H^{k_1}_{\de_1}(\Si)\times H^{k_2}_{\de_2}(\Si)$ to
    $H^k_\de(\Si)$.
  \end{itemize}
\end{lemma}

\subsection{The constraint operator}
\label{sec:constraint-operator}

Using the reduced second fundamental form $\pi$, the
constraint equations \eqref{constraint equations} can be rewritten as
\begin{equation}
  \label{constraint equations 2}
  \begin{aligned}
    R(g)  + \half(\tr_g \pi)^2 - |\pi|^2_g & = 0,
    \\ \div_g \pi  & = 0.
  \end{aligned}
\end{equation}
If $g$ is a Riemannian metric on $\Si$ and $\pi$ is a
symmetric 2-tensor on $\Si$, we define
\begin{align*}
  \HH(g,\pi) & \vcentcolon = R(g) + \half (\tr_{g}\pi)^2 - |\pi|^2_{g},
  \\ \MM(g,\pi) & \vcentcolon = \div_{g}\pi.
\end{align*}
For these operators, we will use the following notation for their
expansion around some given $(g,\pi)$:
\begin{align}
  \HH(g+h,\pi+\varpi) - \HH(g,\pi) & =  D\HH[g,\pi](h,\varpi) + Q_\HH[g,\pi](h,\varpi),\label{expansion H}
  \\ \MM(g+h,\pi+\varpi) - \MM(g,\pi) & = D\MM[g,\pi](h,\varpi) + Q_\MM[g,\pi](h,\varpi),\label{expansion M}
\end{align}
where $D\HH[g,\pi]$ and $D\MM[g,\pi]$ are the linearization of $\HH$
and $\MM$ and where $Q_\HH[g,\pi](h,\varpi)$ and
$Q_\HH[g,\pi](h,\varpi)$ contain all quadratic and higher order terms
in $(h,\varpi)$. Note that in this article, and in particular in the
following lemma, we do not distinguish between coefficients of a
metric $g$ and coefficients of its inverse and thus simply write $g$
in the schematic notation for error terms.

\begin{lemma}
  \label{lem linearized constraint}

  Let $g$ be a Riemannian metric on $\Si$ and $\pi$ a symmetric
  2-tensor on $\Si$. We have
  \begin{align}
    D\HH[g,\pi](h,\varpi)  ={}& - \De_{g}\tr_{g}h  + D^i D^j h_{ij} - h^{ij} Ric(g)_{ij}\label{DH}
    \\&\quad + \tr_{g}\pi\pth{ \tr_{g}\varpi - \pi^{ij} h_{ij}} - 2\pi^{ij} \varpi_{ij} + 2 h^{ij}g^{k\ell}\pi_{ik}\pi_{j\ell} ,\nonumber
    \\ D\MM[g,\pi](h,\varpi)_i  ={}&  \div_{g}\varpi_i  - h^{k\ell}D_{k}\pi_{\ell i} - \pi_i^j g^{k\ell}\pth{   D_k h_{\ell j} - \half D_j h_{k\ell}}   - \half  \pi^{k\ell} D_i h_{k\ell},\label{DM}
  \end{align} 
  where $D$ is the covariant derivative of $g$ and all inverse are
  with respect to $g$. Moreover the terms $Q_\HH[g,\pi](h,\varpi)$ and
  $Q_\MM[g,\pi](h,\varpi)$ have the following schematic form
  \begin{align}
    Q_\HH[g,\pi](h,\varpi)  ={}& h\nab^2 h + (g  \nab g) (h \nab h) + g^2\pth{(\nab h)^2 + \varpi^2}  + \pth{(\nab g)^2 + \pi^2 } h^2 \notag
    \\
                               & +  g\pi h\varpi  + g h(\nab h)^2 + h^2 \nab g\nab h + h^2\pth{ (\nab h)^2 + \varpi^2} + \pi h^2\varpi , \label{QH}
    \\ Q_\MM[g,\pi](h,\varpi) ={}&  g^2 h\nab\varpi + (g\nab g) h \varpi  + g\pi h \nab h  + h^2 \pi \nab g \notag
    \\
                               & + g h \varpi\nab h + h^2\varpi \nab g + h^2\pi \nab h + h^2 \varpi\nab h. \label{QM}
  \end{align}
\end{lemma}

\begin{proof}
The expression \eqref{DH} for $D\HH[g,\pi](h,\varpi)$ follows from the standard linearization of the scalar curvature
\begin{align*}
R(g+h)-R(g) = - \De_{g}\tr_{g}h  + D^i D^j h_{ij} - h^{ij} Ric(g)_{ij} + \text{higher order terms in $h$},
\end{align*}
which can be found in Lemma 2.2 of \cite{Corvino2006}, and the following direct computation
\begin{align*}
\half (\tr_{g+h}(\pi+\varpi))^2 - |\pi+\varpi|^2_{g+h} - \half (\tr_{g}\pi)^2 + |\pi|^2_{g} & =  \tr_g\pi \pth{\tr_g\varpi - h^{ij}\pi_{ij} } - 2 \varpi_{ij}\pi^{ij} +2 \pi_{ik}\pi_{j\ell}g^{ij}h^{k\ell} 
\\&\quad + \text{higher order terms in $h$ and $\varpi$},
\end{align*}
where we used that $(g+h)^{ij}-(g^{ij}-h^{ij})$ is quadratic in $h$. For the expression \eqref{DM} for $D\MM[g,\pi](h,\varpi)_i$, we have
\begin{align*}
\div_{g+h}(\pi+\varpi)_i - \div_{g}\pi_i  & =  (g+h)^{k\ell}D^{(h)}_k\varpi_{i\ell} + (g+h)^{k\ell} D^{(h)}_k\pi_{i\ell}  - g^{k\ell}D_k\pi_{i\ell}  ,
\end{align*}
where $D^{(h)}$ is the covariant derivative of $g+h$. For $T$ a generic symmetric 2-tensor we recall that
\begin{align*}
D^{(h)}_kT_{i\ell} - D_kT_{i\ell} & = \half \pth{ D_k h_{ic} + D_i h_{kc} - D_ch_{ik}} T^c_\ell + \half \pth{ D_k h_{\ell c} + D_\ell h_{kc} - D_ch_{\ell k}} T^c_i 
\\&\quad + \text{higher order terms in $h$}.
\end{align*}
Using this formula we obtain
\begin{align*}
\div_{g+h}&(\pi+\varpi)_i - \div_{g}\pi_i  
\\& = \div_{g}\varpi_i  -h^{k\ell} D^{(h)}_k\pi_{i\ell}  +  g^{k\ell} \pth{ D^{(h)}_k\pi_{i\ell}  - D_k\pi_{i\ell} } + \text{higher order terms in $h$ and $\varpi$}
\\& =  \div_{g}\varpi_i   -h^{k\ell} D_k\pi_{i\ell}  + \half  g^{k\ell} \pth{   \pth{ D_k h_{ic} + D_i h_{kc} - D_ch_{ik}} \pi^c_\ell +  \pth{ D_k h_{\ell c} + D_\ell h_{kc} - D_ch_{\ell k}} \pi^c_i  } 
\\&\quad + \text{higher order terms in $h$ and $\varpi$}
\\& =  \div_{g}\varpi_i   -h^{k\ell} D_k\pi_{i\ell}  + \half  \pth{ D_i h_{kc} \pi^{ck} +  \pth{ 2D_k h_{\ell c} - D_ch_{\ell k}}g^{k\ell}  \pi^c_i  } 
\\&\quad + \text{higher order terms in $h$ and $\varpi$},
\end{align*}
which matches \eqref{DM}. For \eqref{QH} and \eqref{QM}, we use the schematic expressions
\begin{align*}
\HH(g,\pi) & =  g\nab^2g + (g\nab g)^2 +  (g\pi)^2,
\\ \MM(g,\pi) & =  g\pth{\nab\pi + g \pi\nab g }.
\end{align*}
Isolating nonlinear terms in these expressions concludes the proof of the lemma.
\end{proof}

\begin{definition}
  The \emph{constraint operator} is defined from $H^2_{loc}(\Si)\times H^1_{loc}(\Si)$ to $L^2_{loc}(\Si)$ by
  \begin{equation}
    \label{eq:constraint-operator:def}
    \Phi(g,\pi)  \vcentcolon =\pth{ -\HH(g,\pi),2\MM(g,\pi)}.
  \end{equation}
  Its linearization around some given $(g,\pi)$ is denoted by
  $D\Phi[g,\pi]$ and is given by
  \begin{equation}
    \label{eq:linearized-constraint-operator:def}
    D\Phi[g,\pi] (h,\varpi) = \pth{ - D\HH[g,\pi](h,\varpi), 2 D\MM[g,\pi](h,\varpi)}.
  \end{equation}
\end{definition}

Its formal adjoint $D\Phi[g,\pi]^*$ for the $L^2(\Si)$ scalar product
induced by the metric $g$ will play a major role in our construction
since elements of its kernel will be obstructions to solvability. Looking at \eqref{DH}, \eqref{DM} and \eqref{eq:linearized-constraint-operator:def} we see that there exists tensors $A^{(i)}$ depending only on $(g,\pi)$ and their first and second derivatives such that
\begin{align*}
D\Phi[g,\pi] (h,\varpi) & = \pth{ \De_g\tr_g h - \div_g\div_g h + h^{k\ell}A^{(1)}_{k\ell} + \varpi^{k\ell}A^{(2)}_{k\ell}, 2\div_g\varpi_i + h^{k\ell}A^{(3)}_{ ik\ell} + D^a h^{k\ell} A^{(4)}_{ a ik\ell}}.
\end{align*}
By performing formal integration by parts over $\Si$ we then obtain
\begin{align*}
D\Phi[g,\pi]^*(f,X) & = \pth{ \Hess_g(f) - (\De_gf)g + f A^{(1)} - D^a (X^iA^{(4)}_{ ai}) + X^i A^{(3)}_{i}, - \LL_Xg + f A^{(2)} },
\end{align*}
where $\LL_Xg$ is the Lie derivative with respect to $X$. Therefore if $(f,X)$ belongs to $\ker(D\Phi[g,\pi]^*)$ then it solves the system
\begin{align}
\Hess_g(f)  & = B^{(1)}_{ ai} D^a X^i + f B^{(2)}  + X^i B^{(3)}_i , \label{KIDS 1}
\\ \LL_Xg & = f A^{(2)}, \label{KIDS 2}
\end{align}
where $B^{(1)} =  - A^{(1)}  + \half \tr_gA^{(1)} g$, $B^{(2)}=A^{(4)}_{ ai}-\half \tr_gA^{(4)}_{ ai} g$ and $B^{(3)}  = -A^{(3)}_{i} + \half \tr_g A^{(3)}_{i}g + D^a B^{(1)}_{ ai}  $. The exact expressions of $B^{(1)}$, $B^{(2)}$, $B^{(3)}$ and $A^{(2)}$ in terms of $(g,\pi)$ do not matter, and the system \eqref{KIDS 1}-\eqref{KIDS 2} satisfied by the KIDS will only be used in the proof of Lemma \ref{lem regularity kids}.



\section{Statement of the main results}
\label{section main result}

\subsection{The main theorem}

We now state the main result of this article, which handles any asymptotically flat solution of the constraint equations with analyticity near and at the boundary, see Remark \ref{remark more general}. We will then apply it to Kerr in Section \ref{section kerr}.

\begin{theorem}
  \label{theo general}
  Let $(\Si,\bar{g},\bar{\pi})$ be an asymptotically flat initial data
as defined in Definition \ref{def AF}. Assume that there exists $r_1>r_0$
  such that if $\Om\vcentcolon = \Si\cap\{r_0\leq r < r_1\}$,
  then the initial data set
  $\pth{ \Om,\bar{g}_{|_{\Om}},\bar{\pi}_{|_{\Om}} }$ is analytic\footnote{By this we mean analytic in $\{r_0<r<r_1\}$ and on the sphere $\{r=r_0\}$.}. Let $q\in\mathbb{N}^*$
  and $0<\de<1$. There exists $\e_0>0$ only depending on
  $(\Si,\bar{g},\bar{\pi})$, $q$ and $\de$ such that if
  $0<\e\leq \e_0$ then the following holds: for all
  $(\wc{g},\wc{\pi})$ such that
  \begin{align}\label{assumption check}
    \l \wc{g} \r_{H^2_{-q-\de}(\Si)} + \l \wc{\pi} \r_{H^1_{-q-\de-1}(\Si)} \leq \e, 
  \end{align}
  there exists a solution $(g,\pi)$ to the constraint equations
  \eqref{constraint equations 2} on $\Si$ of the form
  \begin{align*}
    g & = u^4 (\bar{g} + \wc{g} + \breve{g}),
    \\ \pi & = u^2(\bar{\pi} + \wc{\pi} + \breve{\pi} + L_{\bar{g} + \wc{g} + \breve{g}}X), 
  \end{align*}
  where
  \begin{equation}
    \label{eq:modified-Lie-derivative:def}
    L_{g}X\vcentcolon = \LL_X g - (\div_{g}X)g,
  \end{equation}
  the scalar function
  $u$ and the vector field $X$ satisfy
  \begin{align*}
    \l u-1 \r_{H^2_{-q-\de}(\Si)} + \l X \r_{H^2_{-q-\de}(\Si)} \leq C \e,
  \end{align*}
  and $(\breve{g},\breve{\pi})$ belong to a vector space of dimension $4q^2$ only
  depending on $(\Si,\bar{g},\bar{\pi})$ and composed of pairs of
  smooth symmetric 2-tensors compactly supported in
  $\Si\cap\{r_0 < r < r_1\}$ and moreover
  \begin{align*}
    \l \breve{g} \r_{W^{2,\infty}(\Si)} + \l \breve{\pi} \r_{W^{1,\infty}(\Si)} \leq C \e,
  \end{align*}
  where $C>0$ only depends on $(\Si,\bar{g},\bar{\pi})$.
\end{theorem}

Theorem \ref{theo general} will be proved in Section \ref{section proof main result} relying on the properties of the linearized operator exhibited in Section \ref{section elliptic}.

\subsection{Spacelike initial data for Kerr stability}\label{section kerr}

In this section we present the main application of Theorem \ref{theo
  general}, i.e. the construction of initial data for the Kerr
stability conjecture. Let $m$ and $a$ be two real numbers representing
the mass and angular momentum of the Kerr black hole respectively. Let
$\mathcal{M}_{m,a}:= \RRR\times (r_+,+\infty)\times (0,\pi) \times
\mathbb{S}^1$ with $r_+= m + \sqrt{m^2-a^2}$. Then, the Kerr metric in
Boyer-Lindquist coordinates $(t,r,\th,\phi)$ is given by
\begin{equation}
  \label{eq:Kerr:Boyer-Lindquist}
  \g_{m,a}  = - \frac{|q|^2\De}{\Si^2} \d t^2 + \frac{\Si^2\sin^2\th}{|q|^2}\pth{\d\phi - \frac{2amr}{\Si^2}\d t}^2 + \frac{|q|^2}{\De}\d r^2 + |q|^2 \d\th^2,
\end{equation}
where
\begin{align*}
  \De & \vcentcolon = r^2 + a^2 - 2mr,\quad q  \vcentcolon = r + i a \cos\th,\quad \Si^2  \vcentcolon = (r^2+a^2)^2 - a^2\sin^2\th \De,
\end{align*}
which is well-defined for\footnote{As usual, the coordinate singularities
  at the north and south poles of the spheres can be resolved by considering for example
  $(x_{\pm},y_{\pm}) = (\sin\theta\cos\phi, \sin\theta\cos\phi)$ as coordinates on neighborhoods of the north and south pole respectively.} $(t,r,\th,\phi)\in \mathcal{M}_{m,a}$.

It is well-known that the singularity of the Kerr metric at the event
horizon $\{r=r_+\}$ is a coordinate singularity, and that the Kerr metric can in fact be
smoothly extended past the event horizon as a Lorentzian metric. For instance, the ingoing Eddington-Finkelstein coordinates $(u_+,r,\th,\phi_+)$ where $\d u_+=\d t + \frac{r^2+a^2}{\De}\d r$ and $\d\phi_+= \d\phi + \frac{a}{\De}\d r$ are such that
\begin{align*}
\g_{m,a} & = - \pth{ 1 - \frac{2mr}{|q|^2}}\d u_+^2 + 2 \d r \d u_+ - \frac{4amr \sin^2\th}{|q|^2} \d u_+ \d \phi_+ - 2a\sin^2\th \d r \d \phi_+  + |q|^2\d\th^2 + \frac{\Si^2}{|q|^2}\sin^2\th \d\phi_+^2,
\end{align*}
see (5.31) in \cite{Hawking1973}. In the next lemma, we define other coordinates $(\tau,r,\th,\ffi)$, also entering the black hole interior, and such that the level surfaces of $\tau$ are spacelike.

\begin{lemma}
  \label{lem kerr}
  There exist coordinates $\tau$ and $\ffi$ such that
  \begin{enumerate}[label=(\roman*)]
  \item in the coordinates $(\tau,r,\th,\ffi)$ the Kerr metric
    $\g_{m,a}$ is smooth on
    \begin{align*}
      \RRR\times [r_+(1-\de_*),+\infty) \times  (0,\pi) \times \mathbb{S}^1 
    \end{align*} 
    and analytic in particular on
    \begin{align*}
      \RRR\times [r_+(1-\de_*),3m) \times (0,\pi) \times \mathbb{S}^1
    \end{align*} 
    for any $0<\de_*<1$,
  \item the hypersurface
    \begin{align*}
      \Si_{m,a}\vcentcolon = \{\tau=0\}\times [r_+(1-\de_*),+\infty) \times (0,\pi) \times \mathbb{S}^1
    \end{align*}
    is uniformly spacelike.
  \end{enumerate}
\end{lemma}

\begin{proof}
We recall the expression of the ingoing principal null frame\footnote{Recall that a null frame satisfies $\g_{m,a}(e_3,e_3)=\g_{m,a}(e_4,e_4)=0$, $\g_{m,a}(e_3,e_4)=-2$ and $(e_1,e_2)$ is an orthonormal basis of $\{e_3,e_4\}^\perp$. The property of the principal null frame of Kerr is that it diagonalizes the curvature tensor. Moreover, the ingoing normalization is such that $e_3$ is geodesic.} for the metric $\g_{m,a}$, see for example Lemma 2.26 in \cite{Klainerman2023}:
  \begin{equation}\label{null frame}
    \begin{aligned}
      e_3 & \vcentcolon = \frac{r^2+a^2}{\De}\dr_t - \dr_r + \frac{a}{\De}\dr_\phi,
      & e_4 & \vcentcolon = \frac{r^2+a^2}{|q|^2}\dr_t + \frac{\De}{|q|^2} \dr_r + \frac{a}{|q|^2} \dr_\phi,
      \\ e_1 & \vcentcolon = \frac{1}{|q|}\dr_\th,
      & e_2 & \vcentcolon = \frac{a\sin\th}{|q|}\dr_t + \frac{1}{|q|\sin\th}\dr_\phi.
    \end{aligned}
  \end{equation}
  The new coordinates $\tau$ and $\ffi$ are defined in terms of the Boyer-Lindquist coordinates by
  \begin{equation}\label{new coord}
    \tau  = t + f(r), \quad \ffi  = \phi + h(r),
  \end{equation}
  where 
  \begin{align*}
    h'(r) & = \frac{a}{\De}, \quad f'(r) = \chi(r) \frac{r^2+a^2}{\De}\sqrt{1-\frac{2m^2\De}{(r^2+a^2)^2}},
  \end{align*}
  with $f(4m)=0$ and where $0\leq\chi\leq 1$ is smooth and such that
  $\chi=1$ for $r<3m$ and $\chi=0$ for $r>4m$. From \eqref{null frame}
  and \eqref{new coord} we first obtain
  \begin{align*}
    & e_3(r)=-1,\quad e_4(r)=\frac{\De}{|q|^2}, \quad e_1(r)=e_2(r)=0,
    \\& e_1(\th)=\frac{1}{|q|}, \quad e_3(\th)=e_4(\th)=e_2(\th) = 0 ,
    \\& e_4(\ffi)= \frac{2a}{|q|^2}, \quad  e_2(\ffi)=\frac{1}{|q|\sin\th}, \quad e_3(\ffi)=e_1(\ffi)= 0,
    \\ & e_3(\tau)  = \frac{r^2+a^2}{\De}\pth{1 - \chi(r) \sqrt{1-\frac{2m^2\De}{(r^2+a^2)^2}}}, \quad e_4(\tau) = \frac{r^2+a^2}{|q|^2}\pth{1 + \chi(r) \sqrt{1-\frac{2m^2\De}{(r^2+a^2)^2}}}, 
    \\ & e_2(\tau) = \frac{a\sin\th}{|q|}, \quad e_1(\tau)  = 0.
  \end{align*}
  Together with the formula for the inverse metric in a coordinate system $(x^\a)_{\a=0,1,2,3}$
  \begin{align*}
    \g^{\a\b}_{m,a} & = - \half e_3(x^\a) e_4(x^\b) - \half e_3(x^\b) e_4(x^\a) + e_1(x^\a)e_1(x^\b) + e_2(x^\a)e_2(x^\b)
  \end{align*}

  from Lemma 4.3 in \cite{Klainerman2023} (which is in fact valid for any Lorentzian metric expressed in a null frame), this
  allows us to obtain the expression of the inverse Kerr metric in the
  new coordinates system $(\tau,r,\th,\ffi)$:
  \begin{align*}
    \g_{m,a}^{rr} & =  \frac{\De}{|q|^2}, \quad \g_{m,a}^{r\ffi}  =  \frac{a}{|q|^2} , \quad \g_{m,a}^{\ffi\ffi}  = \frac{1}{|q|^2\sin^2\th}, \quad \g_{m,a}^{\th\th}  =    \frac{1}{|q|^2},
    \\ \g_{m,a}^{\tau\tau} & = \frac{1}{|q|^2}\pth{ (\chi^2-1) \frac{(r^2+a^2)^2}{\De} - 2m^2\chi^2 + a^2\sin^2\th  },
    \\ \g_{m,a}^{\tau r} & =  \frac{r^2+a^2}{|q|^2} \chi \sqrt{1-\frac{2m^2\De}{(r^2+a^2)^2}} ,
    \\ \g_{m,a}^{\tau\ffi} & =  \frac{a}{|q|^2}\pth{1  + \frac{r^2+a^2}{\De}\pth{ \chi  \sqrt{1-\frac{2m^2\De}{(r^2+a^2)^2}} -1 }},
    \\ \g_{m,a}^{r\th} & = \g_{m,a}^{\th\ffi} =  \g_{m,a}^{\tau\th} =0 .
  \end{align*}
  We see on the above formulas that the inverse Kerr metric is smooth
  on the manifold
  $\RRR\times [r_+(1-\de_*),+\infty) \times (0,\pi) \times
  \mathbb{S}^1$. Moreover, if $r<3m$ then $\chi=1$ and the above
  formulas show that the inverse Kerr metric is actually
  analytic, which shows the first part of the lemma. We also easily check that
  there exists a constant $c_{m,a}>0$ only depending on $m$ and $a$
  such that $\g_{m,a}^{\tau\tau}\leq -c_{m,a}$ which shows the second
  part of the lemma.
\end{proof}

Lemma \ref{lem kerr} implies that the initial data set
$(\Si_{m,a},g_{m,a},\pi_{m,a})$ satisfies the assumptions of Theorem
\ref{theo general} and we immediately deduce the following statement (which is the
precise version of Theorem \ref{theo rough}).

\begin{theorem}
  \label{coro kerr}
  Let $0\leq |a|\le m$ and $(\Si_{m,a},g_{m,a},\pi_{m,a})$ be the
  hypersurface and initial data set constructed in Lemma \ref{lem
    kerr}. Let $q\in \mathbb{N}^*$ and $0<\de<1$. There exists
  $\e_0=\e_0(m,a,q,\de)>0$ such that if $0<\e\leq\e_0$ then the
  following holds: for all $(\wc{g},\wc{\pi})$ with
  \begin{align*}
    \l \wc{g}\r_{H^2_{-q-\de}(\Si_{m,a})} + \l \wc{\pi} \r_{H^1_{-q-\de}(\Si_{m,a})} \leq \e,
  \end{align*}
  there exists a solution $(g,\pi)$ of the constraint equations \eqref{constraint equations 2} on $\Si_{m,a}$ of the form
  \begin{align*}
    g & = u^4\pth{ g_{m,a} + \wc{g} + \breve{g}},
    \\ \pi & = u^2 \pth{ \pi_{m,a} + \wc{\pi} + \breve{\pi} + L_{g_{m,a}+ \wc{g} + \breve{g} }X},
  \end{align*}
  where the correctors $(\breve{g},\breve{\pi})$ belong to a vector space of dimension $4q^2$ only
  depending on $(m,a)$ and composed of pairs of
  smooth symmetric 2-tensors compactly supported and where $u$, $X$ and $(\breve{g},\breve{\pi})$ satisfy
  \begin{align*}
    \l u-1\r_{H^2_{-q-\de}(\Si_{m,a})} + \l X \r_{H^2_{-q-\de}(\Si_{m,a})} + \l \breve{g} \r_{W^{2,\infty}(\Si_{m,a})} + \l \breve{\pi} \r_{W^{1,\infty}(\Si_{m,a})} \lesssim \e.
  \end{align*}
\end{theorem}

\begin{remark}
We remark that \Cref{coro kerr} is true even in the extremal case where $|a|=m$. Despite the horizon linear instabilities (see \cite{Aretakis2015,Gajic2023}) displayed by extremal black holes, the codimension 1 nonlinear stability of extremal Kerr is conjectured to be true in \cite{Dafermos2024} (see the recent \cite{Angelopoulos2024} for a proof of the analogue result for Reissner-Nordström in spherical symmetry). Our result thus provides spacelike initial data for a potential future proof of this conjecture.   
\end{remark}

\section{Properties of the linearized operator}
\label{section elliptic}

Our first task is to develop a spectral theoretic framework for the following
elliptic operator on $\Si$ acting on pairs $\mathbb{U}=(u,X)$
composed of a scalar function $u$ and a vector field $X$:
\begin{equation}
  \label{def P}
  \begin{aligned}
    P(\mathbb{U}) & \vcentcolon = \pth{-8\De_{\bar{g}} u   - 2\bar{\pi}^{ij} \LL_{X} \bar{g}_{ij} + \tr_{\bar{g}}\bar{\pi} \div_{\bar{g}}X,
                    -2\div_{\bar{g}}L_{\bar{g}}X_i  - 8  \bar{g}^{j\ell} \bar{\pi}_{i\ell} \dr_j u   + 4   \tr_{\bar{g}}\bar{\pi}  \dr_i u }.
  \end{aligned}
\end{equation}
This operator will naturally appear after we formulate the constraint
equations following the conformal ansatz proposed in
\cite{Corvino2006}, see Section \ref{section
  conformal formulation}. For any $\mathbb{U}=(u,X)$, we will also
define\footnote{Directly integrating by parts $\div_{\bar{g}}L_{\bar{g}}X_i$ would lead to the boundary term $L_{\bar{g}}X(\nu,\cdot)$. Here we instead first expand $\div_{\bar{g}}L_{\bar{g}}X_i=\De^{(1)}_{\bar{g}}X_i + Ric(\bar{g})(\dr_i,X)$ and then integrate by parts to obtain the boundary term $D_\nu X$.}
\begin{align*}
  P^*(\mathbb{U}) & \vcentcolon= \pth{ -8\De_{\bar{g}} u + 4 \bar{\pi}^{ij} \LL_X\bar{g}_{ij} - 4 \div_{\bar{g}}(\tr_{\bar{g}}\bar{\pi} X) , - 2\div_{\bar{g}}L_{\bar{g}}X_i + 4  \bar{\pi}_{i}^j \dr_j u - \dr_i( \tr_{\bar{g}}\bar{\pi} u) },
  \\ B_\nu(\mathbb{U}) & \vcentcolon= \pth{ - 8 \dr_\nu u  + \tr_{\bar{g}}\bar{\pi} \bar{g}(X,\nu) - 4 \bar{\pi}(X,\nu) , -2 D_\nu X } ,
  \\ B_\nu^*(\mathbb{U}) & \vcentcolon=  \pth{ - 8 \dr_\nu u  -4 \tr_{\bar{g}}\bar{\pi} \bar{g}(X,\nu) +8 \bar{\pi}(X,\nu) , -2 D_\nu X } ,
\end{align*}
with $D$ being the covariant derivative associated with $\bar{g}$ so that, for $\mathbb{U}=(u,X)$ and $\mathbb{V}=(v,Y)$ sufficiently regular, we have
\begin{align}\label{IPP}
  \int_{\Si} \pth{ P(\mathbb{U})\cdot \mathbb{V} - \mathbb{U}\cdot P^*(\mathbb{V}) } + \int_{\dr\Si}\pth{ B_\nu(\mathbb{U})\cdot \mathbb{V} - \mathbb{U}\cdot B_\nu^*(\mathbb{V}) } = 0,
\end{align}
where $\mathbb{U}\cdot \mathbb{V}=uv + \bar{g}(X,Y)$ and $\nu$ denotes the inward pointing
unit normal to $\dr\Si$ with respect to $\bar{g}$ and the integration is made with respect to the volume form of $\bar{g}$.

We will consider perturbations of the boundary operators $B_\nu$ and $B_\nu^*$. For that, we rely on the following intrinsic definition.
\begin{definition}
We introduce the following:
\begin{itemize}
\item[(i)] For $b$ a scalar function, $\om$ a 1-form, $Z$ a vector field and $\xi$ a $(1,1)$-tensor, we define $F=(b,\om,Z,\xi)$ so that $F$ acts on the space of pairs $\mathbb{U}=(u,X)$ of functions and vector fields through the formula
\begin{align*}
    F(\mathbb{U})=\pth{b u + \om(X), u Z + \xi(X)}.
\end{align*}
We define the space of such smooth $F$ restricted to $\dr\Si$ by $\EE(\dr\Si)$.
\item[(ii)] If $F\in \EE(\dr\Si)$ we define its transpose $F^*\in \EE(\dr\Si)$ by the formula $F^*=\pth{b,Z^\flat,\om^\sharp,\xi^\natural}$, where $\xi^\natural$ is defined in coordinates by $(\xi^\natural)_i^{\;\;j}=\bar{g}^{j\ell} \xi_{\ell}^{\;\;k} \bar{g}_{ki}$. This definition is such that 
\begin{align}\label{duality}
F(\mathbb{U})\cdot \mathbb{V} = \mathbb{U} \cdot F^*(\mathbb{V}).
\end{align}
\item[(iii)] We define $\mathrm{Id}\in \EE(\dr\Si)$ by $\mathrm{Id}=(1,0,0,\de)$ so that $\mathrm{Id}(\mathbb{U})=\mathbb{U}$.
\end{itemize}
\end{definition}

For $F=(b,\om,Z,\xi)\in\EE(\dr\Si)$ we use the notation $\l F\r_{\mathcal{X}} \vcentcolon =\l b\r_{\mathcal{X}}+\l \om\r_{\mathcal{X}}+\l Z\r_{\mathcal{X}}+\l \xi\r_{\mathcal{X}}$ where $\mathcal{X}$ is any relevant function space defined on $\dr\Si$. We are now ready to define properly our elliptic boundary operators with precise domains.

\begin{definition}
  For $F\in \EE(\dr\Si)$ and $\de\in\RRR$ we define 
  \begin{align}
    P_{F,\de} & \vcentcolon = \pth{ (P,B_\nu + F ): H^2_\de(\Sigma)\longrightarrow L^2_{\de-2}(\Sigma)\times H^{\half}(\dr\Si) }, \label{eq:P-f-delta:def}              
    \\
    P^*_{F,\de} & \vcentcolon = \pth{ (P^*,B^*_\nu + F^* ): H^2_\de(\Sigma)\longrightarrow L^2_{\de-2}(\Sigma)\times H^{\half}(\dr\Si) }. \label{eq:P-f-delta-adjoint:def}
  \end{align}  
\end{definition}

Beware that the operator $P^*_{F,\de}$ defined above is not the adjoint of $P_{F,\de}$. See the discussion before Lemma \ref{lem weird}.

\subsection{Basic properties}

We gather in the following proposition general facts about the
operators \eqref{eq:P-f-delta:def} and
\eqref{eq:P-f-delta-adjoint:def}, which follow from $P_{F,\de}$ and
$P^*_{F,\de}$ being compact perturbations of the Laplace-Beltrami
operator (acting on both scalar functions and vector fields) with
Neumann boundary conditions.

\begin{proposition}\label{prop fredholm}
Let $F\in \EE(\dr\Si)$.
\begin{itemize}
\item[(i)] If $\de\notin\mathbb{Z}$, then there exists $S_\de\subset \Si$ bounded such that for all $\mathbb{U}\in H^2_\de(\Si)$ we have
\begin{equation}\label{estimation elliptique}
\begin{aligned}
\l \mathbb{U} \r_{H^2_\de(\Si)} & \lesssim \l P(\mathbb{U}) \r_{L^2_{\de-2}(\Si)} + \l (B_\nu +F)(\mathbb{U})\r_{H^\half(\dr\Si)} + \l \mathbb{U} \r_{L^2(S_\de)},
\\ \l \mathbb{U} \r_{H^2_\de(\Si)} & \lesssim \l P^*(\mathbb{U}) \r_{L^2_{\de-2}(\Si)} + \l (B^*_\nu +F^*)(\mathbb{U})\r_{H^\half(\dr\Si)} + \l \mathbb{U} \r_{L^2(S_\de)},
\end{aligned}
\end{equation}
\item[(ii)] If $\de\notin\mathbb{Z}$, then the operators $P_{F,\de}$ and $P^*_{F,\de}$ are semi-Fredholm.
\item[(iii)] If $-1<\de<0$, then the operators $P_{F,\de}$ and $P^*_{F,\de}$ are Fredholm with vanishing index.
\end{itemize}
\end{proposition}

\begin{proof}
Consider first the Neumann Laplace-Beltrami operator of $\bar{g}$ defined as follows:
\begin{align*}
\De_{\bar{g},N,\de} : H^2_{\delta}(\Sigma)&\longrightarrow L^2_{\delta-2}(\Sigma)\times H^{\frac{1}{2}}(\partial\Sigma)
\\ u &\longmapsto \pth{ \De_{\bar{g}}u,\partial_{\nu}u}.
\end{align*}
Let $0\leq\chi\leq 1$ be a smooth cutoff function such that $\chi=1$ in a neighborhood of $\dr\Si$ and such that $S\vcentcolon = \mathrm{supp}(\chi)$ is compact in $\Si$. Let $u\in H^2_\de(\Si)$. From Proposition 4 in \cite{Maxwell2005} we have
\begin{align*}
\l \chi u \r_{H^2_\de(\Si)} \lesssim \l \De_{\bar{g}} (\chi u) \r_{L^2(S)} + \l \dr_\nu (\chi u) \r_{H^{\frac{3}{2}}(\dr\Si)} + \l \chi u \r_{L^2(S)}
\end{align*}
and from Theorem 1.10 in \cite{Bartnik1986} (which requires $\de\notin\mathbb{Z}$) we have
\begin{align*}
\l (1-\chi) u \r_{H^2_\de(\Si)} \lesssim \l \De_{\bar{g}} ((1-\chi)u) \r_{L^2_{\de-2}(\Si)} + \l (1-\chi) u \r_{L^2(S')},
\end{align*}
where $S'$ is some bounded open subset of $\Si$ and we can assume that $S\subset S'$. Since $0\leq \chi \leq 1$ and $\chi=1$ in a neighborhood of $\dr\Si$ we have $\l \dr_\nu (\chi u) \r_{H^{\frac{3}{2}}(\dr\Si)}=\l \dr_\nu u \r_{H^{\frac{3}{2}}(\dr\Si)}$ and $\l \chi u \r_{L^2(S)}+\l (1-\chi) u \r_{L^2_\de(S')}\lesssim \l  u \r_{L^2_\de(S')}$. Moreover, expanding $\De_{\bar{g}} (\chi u)$ and $\De_{\bar{g}} ((1-\chi) u)$ and applying the interpolation result from Theorem 3.5 in Appendix I of \cite{ChoquetBruhat2009} (with $p=q=2$, $j=1$ and $m=2$) to the first order term in $u$ (which are supported in $S'$) we obtain
\begin{align*}
\l u \r_{H^2_\de(\Si)} \lesssim \l \De_{\bar{g}} u \r_{L^2_{\de-2}(\Si)}  + \mu \l u \r_{H^2_\de(\Si)}  +  \l \dr_\nu u \r_{H^{\frac{3}{2}}(\dr\Si)}+  C_\mu \l u \r_{L^2(S')}
\end{align*}
for all $\mu>0$. Choosing $\mu$ small enough in order to absorb $\mu \l u \r_{H^2_\de(\Si)}$ into the left hand side shows that
\begin{align*}
\l u \r_{H^2_\de(\Si)} \lesssim \l \De_{\bar{g},N,\de} (u) \r_{L^2_{\de-2}(\Si)\times H^{\frac{3}{2}}(\dr\Si)}  +  \l u \r_{L^2(S')}.
\end{align*}
As is standard, this inequality is sufficient to show that $\De_{\bar{g},N,\de}$ is a semi-Fredholm operator for all $\de\notin\mathbb{Z}$. Moreover, if $-1<\de<0$, Proposition 1 in \cite{Maxwell2005} shows that it is Fredholm with vanishing index. We have thus proved the proposition for $\De_{\bar{g},N,\de}$. However, using again Theorem 3.5 in Appendix I of \cite{ChoquetBruhat2009}, the standard trace embedding $ H^1(\Si)\xhookrightarrow{}H^\half(\dr\Si)$ for the lower order terms and the first property in Lemma \ref{lem embedding} we see that $P_{F,\de}$ and $P^*_{F,\de}$ are compact perturbations of $\De_{\bar{g},N,\de}$ (when considered as acting on a scalar function and the three components of a vector field). Therefore the properties of $\De_{\bar{g},N,\de}$ naturally extend to $P_{F,\de}$ and $P^*_{F,\de}$ and the proposition is proved.
\end{proof}

The adjoint $(P_{F,\de})^*$ of $P_{F,\de}$ is defined for all $(\mathbb{V},\mathbbm{h})\in L^2_{-1-\de}(\Si)\times H^{-\half}(\dr\Si)$ and for all $\mathbb{U}\in H^2_\de(\Si)$ by the relation
\begin{align}\label{def adjoint}
\left[(P_{F,\de})^*(\mathbb{V},\mathbbm{h})\right](\mathbb{U}) & \vcentcolon =  \ps{P(\mathbb{U}), \mathbb{V}}_{L^2_{\de-2}(\Si)\times L^2_{-1-\de}(\Si)} + \ps{ (B_\nu +F)(\mathbb{U}),\mathbbm{h}}_{H^\half(\dr\Si)\times H^{-\half}(\dr\Si)}.
\end{align}
The adjoint $\pth{P^*_{F,\de}}^*$ of $P^*_{F,\de}$ is defined
similarly. In order to characterize the mapping properties of
$P_{F,\delta}$ and $P_{F,\delta}^{*}$, we will use the following lemma
that characterizes the kernels of their adjoints.

\begin{lemma}
  \label{lem weird}
  If $\de\notin\mathbb{Z}$ and $F\in \EE(\dr\Si)$ then
  \begin{align}
    \ker\pth{ \pth{P_{F,\de}}^* } & = \enstq{\pth{\mathbb{V},\mathbb{V}_{|_{\dr\Si}}}}{\mathbb{V}\in \ker\pth{P^*_{F,-1-\de}}}, \label{kernel P^*}
    \\ \ker\pth{ \pth{P^*_{F,\de}}^* }& = \enstq{\pth{\mathbb{V},\mathbb{V}_{|_{\dr\Si}}}}{\mathbb{V}\in \ker\pth{P_{F,-1-\de}}} . \label{kernel P^**}
  \end{align}
\end{lemma}

\begin{proof}
We will only prove \eqref{kernel P^*} since \eqref{kernel P^**} follows by a similar argument. For the inclusion from right to left in \eqref{kernel P^*}, let $\mathbb{V}\in H^2_{-1-\de}(\Si)$. For any $\mathbb{U}\in H^2_\de(\Si)$ the integration by parts \eqref{IPP} is justified and together with \eqref{duality} and \eqref{def adjoint} we obtain
\begin{align*}
\left[(P_{F,\de})^*\pth{\mathbb{V},\mathbb{V}_{|_{\dr\Si}}}\right](\mathbb{U}) & = \ps{P(\mathbb{U}), \mathbb{V}}_{L^2_{\de-2}(\Si)\times L^2_{-1-\de}(\Si)} + \ps{ (B_\nu +F)(\mathbb{U}),\mathbb{V}}_{H^\half(\dr\Si)\times H^{-\half}(\dr\Si)}
\\& =  \ps{\mathbb{U}, P^*(\mathbb{V})}_{L^2_{\de}(\Si)\times L^2_{-3-\de}(\Si)}  + \ps{ \mathbb{U},(B^*_\nu +F^*)\pth{\mathbb{V}}}_{H^\half(\dr\Si)\times H^{-\half}(\dr\Si)}.
\end{align*}
If moreover $\mathbb{V}$ is assumed to satisfy $P^*(\mathbb{V})=0$ in $\Si$ and $(B^*_\nu +F^*)\pth{\mathbb{V}}=0$ in $\dr\Si$ the above expression vanishes and we get $(P_{F,\de})^*\pth{\mathbb{V},\mathbb{V}_{|_{\dr\Si}}}=0$.

For the inclusion from left to right in \eqref{kernel P^*}, let $(\mathbb{V},\mathbbm{h})\in L^2_{-1-\de}(\Si)\times H^{-\half}(\dr\Si)$ such that 
\begin{align}\label{hypothèse}
\ps{P(\mathbb{U}), \mathbb{V}}_{L^2_{\de-2}(\Si)\times L^2_{-1-\de}(\Si)} + \ps{ (B_\nu +F)(\mathbb{U}),\mathbbm{h}}_{H^\half(\dr\Si)\times H^{-\half}(\dr\Si)} & = 0,
\end{align}
for all $\mathbb{U}\in H^2_\de(\Si)$. We first apply \eqref{hypothèse} for every $\mathbb{U}\in C^\infty_c(\Si\setminus\dr\Si)$ smooth and compactly supported in $\Si\setminus\dr\Si$. Since the boundary term in \eqref{hypothèse} vanishes in this case, we get $\ps{P(\mathbb{U}), \mathbb{V}}_{L^2_{\de-2}(\Si)\times L^2_{-1-\de}(\Si)}=0$ for all $\mathbb{U}\in C^\infty_c(\Si\setminus\dr\Si)$, hence $P^*(\mathbb{V})=0$. In view of the first part of Theorem 6.5 in Chapter 2 of \cite{Lions1972}\footnote{\label{footnote10}More specifically, the correspondence between our notations and the ones of Theorem 6.5 in Chapter 2 of \cite{Lions1972} is $A=P^*$, $r=0$, $\Om=\Si$ and $(B_0,B_1)(\mathbb{V})=\pth{\mathbb{V}_{|_{\dr\Si}},(B^*_\nu +F^*)\pth{\mathbb{V}}}$. While this result is stated for scalar operators, note that it extends immediately to principally scalar operators such as the ones we consider.}, the fact that $P^*(\mathbb{V})=0$ with $\mathbb{V}\in L^2_{-1-\de}(\Si)$ allows us to get the trace properties $\mathbb{V}\in H^{-\half}(\dr\Si)$ and $(B^*_\nu + F^*)(\mathbb{V})\in H^{-\frac{3}{2}}(\dr\Si)$. Now if $\mathbb{U}\in H^2_\de(\Si)$, the second part of Theorem 6.5 in Chapter 2 of \cite{Lions1972}\footnote{In addition to the correspondence of Footnote \ref{footnote10}, we also have $A^*=P$ and $(T_0,T_1)(\mathbb{U})=\pth{\mathbb{U}_{|_{\dr\Si}},(B_\nu + F)\pth{\mathbb{U}}}$.} gives
\begin{align*}
  \ps{P(\mathbb{U}),\mathbb{V}}_{L^2_{\de-2}(\Si)\times L^2_{-1-\de}(\Si)}  = - \ps{(B_\nu+F)(\mathbb{U}),\mathbb{V}}_{H^\half(\dr\Si)\times H^{-\half}(\dr\Si)} + \ps{\mathbb{U},(B_\nu^*+F^*)(\mathbb{V})}_{H^{\frac{3}{2}}(\dr\Si)\times H^{-\frac{3}{2}}(\dr\Si)}
\end{align*}
where we used $P^*(\mathbb{V})=0$. Using now \eqref{hypothèse} we deduce 
\begin{align}\label{machin}
\ps{(B_\nu+F)(\mathbb{U}),\mathbb{V}-\mathbbm{h}}_{H^\half(\dr\Si)\times H^{-\half}(\dr\Si)} & = \ps{\mathbb{U},(B_\nu^*+F^*)(\mathbb{V})}_{H^{\frac{3}{2}}(\dr\Si)\times H^{-\frac{3}{2}}(\dr\Si)}
\end{align}
for all $\mathbb{U}\in H^2_\de(\Si)$. In view of the trace lifting properties found in Theorem 8.3 in Chapter 1 of \cite{Lions1972}, the map $B_\nu+ F : \enstq{\mathbb{U}\in H^2_\de(\Si)}{\mathbb{U}_{|_{\dr\Si}}=0} \longrightarrow H^\half(\dr\Si)$ is surjective. Therefore, \eqref{machin} implies $\mathbbm{h}=\mathbb{V}_{|_{\dr\Si}}$. Similarly, the surjectivity of the trace operator from $H^2_\de(\Si)$ to $H^{\frac{3}{2}}(\dr\Si)$ (see again Theorem 8.3 in Chapter 1 of \cite{Lions1972}) and \eqref{machin} again implies $(B_\nu^*+F^*)(\mathbb{V})=0$. 

It only remains to prove that $\mathbb{V}\in H^2_{-1-\de}(\Si)$. For that, we adapt the strategy of Proposition 5.4 in Chapter 2 of \cite{Lions1972}. More precisely, we introduce the following subset of $L^2_{-1-\de}(\Si)$:
\begin{align}\label{def Nstar}
N^* \vcentcolon = \enstq{\mathbb{Y}\in H^2_{-1-\de}(\Si)}{P^*(\mathbb{Y})=0, \; (B_\nu^*+F^*)(\mathbb{Y}) =0}.
\end{align}
Note that by the elliptic estimate \eqref{estimation elliptique}, $N^*$ is a closed subset of $L^2_{-1-\de}(\Si)$. We denote by $\Pi_{N^*}$ the orthogonal projection on $N^*$ for the scalar product $\ps{\ps{\mathbb{Y},\mathbb{W}}}\vcentcolon=\int_\Si r^{2\de-1}\mathbb{Y}\cdot\mathbb{W}$ defined on $L^2_{-1-\de}(\Si)$. We now decompose $\mathbb{V}=\mathbb{V}_1 + \mathbb{V}_2$ with $\mathbb{V}_1\in N^*$ and $\mathbb{V}_2\in L^2_{-1-\de}(\Si)$ satisfying
\begin{align}\label{orthogonalité}
\ps{\ps{\mathbb{V}_2,\mathbb{W}}}=0, \quad \text{for all $\mathbb{W}\in N^*$}.
\end{align}
Our goal is now to show that $\mathbb{V}_2=0$. For any $\mathbb{H}\in L^2_{-1-\de}(\Si)$, Lemma \ref{lemma lax milgram} below gives the existence of some $\mathbb{U}_{\mathbb{H}}\in H^2_\de(\Si)$ such that 
\begin{align*}
P\pth{\mathbb{U}_{\mathbb{H}}} & = r^{2\de-1} (\mathrm{Id}-\Pi_{N^*})(\mathbb{H}) \quad \text{and} \quad (B_\nu + F)(\mathbb{H}) = 0.
\end{align*}
The second part of Theorem 6.5 in Chapter 2 of \cite{Lions1972} applied to $\mathbb{U}_{\mathbb{H}}$ and $\mathbb{V}_2$ (which satisfies $P^*(\mathbb{V}_2)=0$ and $(B_\nu^*+F^*)(\mathbb{V}_2)=0$ since $\mathbb{V}_2=\mathbb{V}-\mathbb{V}_1$) becomes $\ps{P(\mathbb{U}_{\mathbb{H}}), \mathbb{V}_2}_{L^2_{\de-2}(\Si)\times L^2_{-1-\de}(\Si)}=0$ and thus $\ps{\ps{(\mathrm{Id}-\Pi_{N^*})(\mathbb{H}),\mathbb{V}_2}}=0$. Thanks to \eqref{orthogonalité}, this becomes $\ps{\ps{\mathbb{H},\mathbb{V}_2}}=0$. Since this is true for every $\mathbb{H}\in L^2_{-1-\de}(\Si)$, we have indeed proved that $\mathbb{V}_2=0$. Therefore, we have obtained $\mathbb{V}\in N^* \subset H^2_{-1-\de}(\Si)$, which concludes the proof of the lemma.
\end{proof}

We now prove an auxiliary technical lemma used in the previous proof.

\begin{lemma}\label{lemma lax milgram}
For every $\mathbb{H}\in L^2_{-1-\de}(\Si)$, there exists $\mathbb{U}_{\mathbb{H}}\in H^2_\de(\Si)$ such that 
\begin{align*}
P\pth{\mathbb{U}_{\mathbb{H}}} & = r^{2\de-1} (\mathrm{Id}-\Pi_{N^*})(\mathbb{H}) \quad \text{and} \quad (B_\nu + F)(\mathbb{H}) = 0,
\end{align*}
where $\Pi_{N^*}$ is defined in the proof of Lemma \ref{lem weird}.
\end{lemma}

\begin{proof}
We consider the space $M^*\vcentcolon = \enstq{\mathbb{V}\in H^2_{-1-\de}(\Si)}{ (B_\nu^*+F^*)(\mathbb{V}) = 0}$ and define on it the continuous quadratic form
\begin{align*}
q(\mathbb{V},\mathbb{W}) = \int_{\Si} r^{2\de+3} P^*(\mathbb{V})\cdot P^*(\mathbb{W}).
\end{align*}
With $N^*$ defined in \eqref{def Nstar}, we use the notation $(N^*)^{\perp_{M^*}} \vcentcolon = (N^*)^\perp\cap M^*$. Using the semi-Fredholm estimate, the quadratic form $q$ is coercive on $(N^*)^{\perp_{M^*}}$. Therefore for any $\mathbb{H}\in L^2_{-1-\de}(\Si)$, we may apply Lax-Milgram's theorem which yields the existence of a unique $\mathbb{V}_{\mathbb{H}}\in (N^*)^{\perp_{M^*}}$ such that 
\begin{align*}
q(\mathbb{V}_{\mathbb{H}},\mathbb{W})= \ps{\ps{ \mathbb{H},\mathbb{W} }}, \quad \text{for any $\mathbb{W}\in(N^*)^{\perp_{M^*}}$.}
\end{align*}
Recalling the definition of $\Pi_{N^*}$, we infer $q(\mathbb{V}_{\mathbb{H}},\mathbb{W})= \ps{\ps{ (\mathrm{Id}-\Pi_{N^*})(\mathbb{H}),\mathbb{W} }}$ for any $\mathbb{W}\in(N^*)^{\perp_{M^*}}$. Moreover, since for any $\mathbb{W}\in N^*$ we have $q(\mathbb{V}_{\mathbb{H}},\mathbb{W})=0$ and $\ps{\ps{ (\mathrm{Id}-\Pi_{N^*})(\mathbb{H}),\mathbb{W} }}=0$, we obtain
\begin{align}\label{LM}
q(\mathbb{V}_{\mathbb{H}},\mathbb{W})= \ps{\ps{ (\mathrm{Id}-\Pi_{N^*})(\mathbb{H}),\mathbb{W} }}, \quad \text{for any $\mathbb{W}\in M^*$.}
\end{align}
By considering first $\mathbb{W}$ smooth and compactly supported in $\Si\setminus \dr\Si$ and then compactly supported in $\Si$ and belonging to $M^*$, we deduce by integration by parts in \eqref{LM} that
\begin{align}\label{4thorder}
P(r^{2\de+3}P^*(\mathbb{V}_{\mathbb{H}}))=r^{2\de-1}(\mathrm{Id}-\Pi_{N^*})(\mathbb{H}), \quad (B_\nu + F)(r^{2\de+3}P^*(\mathbb{V}_{\mathbb{H}})) = 0, \quad (B_\nu^* + F^*)(\mathbb{V}_{\mathbb{H}}) = 0.
\end{align}
To upgrade the regularity of $\mathbb{V}_{\mathbb{H}}$ from $H^2_{-1-\de}(\Si)$ to $H^4_{-1-\de}(\Si)$, we apply two different elliptic regularity results to the 4-th order elliptic boundary value problem appearing in \eqref{4thorder}: to $(1-\chi)\mathbb{V}_{\mathbb{H}}$ we apply twice Proposition 1.6 from \cite{Bartnik1986} and to $\chi \mathbb{V}_{\mathbb{H}}$ we apply Proposition 5.1 from Chapter 2 of \cite{Lions1972}, where $\chi$ is as in the proof of Proposition \ref{prop fredholm}. We thus obtain $\mathbb{V}_{\mathbb{H}}\in H^4_{-1-\de}(\Si)$ and setting $\mathbb{U}_{\mathbb{H}}\vcentcolon = r^{2\de+3}P^*(\mathbb{V}_{\mathbb{H}})$ concludes the proof of the lemma.
\end{proof}

In the next lemma, we give a unique continuation result from the boundary for $P$ and $P^*$. 

\begin{lemma}
  \label{lem unique continuation}
  Let $\Ga$ be an open subset of $\dr\Si$. If $\mathbb{U}\in H^2_{loc}(\Si)$ satisfies $P(\mathbb{U})=0$ (resp. $P^*(\mathbb{U})=0$) on $\Si$ and $B_\nu(\mathbb{U})=\mathbb{U}=0$ (resp. $B^*_\nu(\mathbb{U})=\mathbb{U}=0$) on $\Ga$, then $\mathbb{U}=0$ on $\Si$.
\end{lemma}

\begin{proof}
Lemma \ref{lem unique continuation} is the extension to systems of the corresponding statement for second order elliptic scalar operators on manifold in Theorem 9.8 in \cite{LeRousseau2022}. This theorem is an immediate corollary of the Carleman estimates proved in Theorem 8.30 in \cite{LeRousseau2022} (itself being deduced from the corresponding Euclidean Carleman estimates of Theorem 3.28 in \cite{LeRousseau2022a}). Since $P$ and $P^*$ are diagonal operators at leading order, we can sum the Carleman estimates given in the statement of Theorem 8.30 in \cite{LeRousseau2022} applied to each of the four components of $\mathbb{U}$ and absorb the lower order coupling terms by playing with the large constant $\tau$ in this statement.
\end{proof}

\subsection{A special boundary condition}
\label{section BC}

In this section, we will construct a distinguished boundary condition
so that $P_{F,\delta}$ as defined in \eqref{def P} has good mapping properties for any fixed $\de\in(0,1)$.
\begin{lemma}
  \label{lem bartnik}
  Let $F_0\in \EE(\dr\Si)$. There exists
  $\eta(F_0)>0$ only depending on $F_0$ such that for all $F\in \EE(\dr\Si)$
  on $\dr\Si$,
  \begin{align*}
    \l F - F_0 \r_{H^{\frac{1}{2}}(\dr\Si)}\leq \eta(F_0) \Longrightarrow \dim(\ker(P_{F,-\de})) \leq \dim(\ker(P_{F_0,-\de})).
  \end{align*}
\end{lemma}
\begin{proof}
    \Cref{lem bartnik} follows immediately from Proposition 1.11 in
    \cite{Bartnik1986} and the fact that
    \begin{align*}
        \l P_{F,-\de} - P_{F_0,-\de}\r_{\mathrm{op}} \lesssim \l F-F_0 \r_{H^{\frac{1}{2}}(\dr\Si)},
    \end{align*}
    where $\norm*{\cdot}_{\mathrm{op}}$ refers to the standard operator
    norm.
\end{proof}

In the next lemma, we show how to lower the dimension of the kernel in
a constructive way.

\begin{lemma}
\label{lem bartnik 2}
Let $\de\in(0,1)$. Let $\mathcal{W}$ be an open subset of $\dr\Si$ such that $\dr\Si\setminus \mathcal{W}$ contains an open set. Let $F_0\in \EE(\dr\Si)$ such that
  $\dim(\ker(P_{F_0,-\de}))>0$. Then there exists $F_1\in \EE(\dr\Si)$ such that 
  \begin{equation}
    \label{eq:dimension-reducing-boundary-condition-change}
    \dim(\ker(P_{F_1,-\de})) < \dim(\ker(P_{F_0,-\de})), \quad \text{and} \quad (F_1-F_0)_{|_\mathcal{W}}=0.
  \end{equation}
\end{lemma}

\begin{proof}
  Using \Cref{prop fredholm}, the Fredholm index of $P_{F_0,-\de}$ is
  zero for $\de\in(0,1)$. Thus, the assumption that
  $\dim(\ker(P_{F_0,-\de}))>0$ implies that
  $\dim\pth{\ker\pth{(P_{F_0,-\de})^*}}>0$. According to \eqref{kernel
    P^*} this implies that
  $\dim\pth{\ker\pth{P_{F_0,-1+\de}^*}}>0$. Now consider
  $(\mathbb{V}_0,\mathbb{V}_0^*)\in \ker(P_{F_0,-\de})\times
  \ker(P_{F_0,-1+\de}^*)$ such that $\mathbb{V}_0, \mathbb{V}_0^*$ are
  both non-zero. Since they satisfy $(P,B_\nu+F_0)(\mathbb{V}_0)=0$
  and $(P^*,B^*_\nu+F_0^*)(\mathbb{V}_0^*)=0$, unique continuation for
  the operators $P$ and $P^*$ (recall Lemma \ref{lem unique continuation}) implies that 
  \begin{align*}
      \{ \mathbb{V}_0\neq 0\} \cap \{ \mathbb{V}_0^*\neq 0\} \cap \dr\Si
  \end{align*}
  is a dense open subset of $\dr\Si$. Thus, there exists $H_0\in \EE(\dr\Si)$ such that $H_{0|_{\mathcal{W}}}=0$ and
  \begin{align}\label{def h zero}
    \int_{\dr\Si}H_0 (\mathbb{V}_0) \cdot \mathbb{V}_0^* = 1. 
  \end{align}
  For $\eta>0$, we define
  \begin{align*}
    F_\eta\vcentcolon=F_0 + \eta \frac{H_0}{\l H_0 \r_{H^{\frac{3}{2}}(\dr\Si)}}.   
  \end{align*}
  We will show that there exists some $\eta$ sufficiently small such
  that
  $\dim\left(\ker (P_{F_{\eta},-\delta})\right)<\dim\left(\ker
    (P_{F_0,-\delta})\right)$.  For $\eta\leq \eta(F_0)$, Lemma
  \ref{lem bartnik} already implies that
  $\dim(\ker(P_{F_\eta,-\de})) \leq \dim(\ker(P_{F_0,-\de}))$.  Now
  assume for the sake of contradiction that
  $\dim(\ker(P_{F_\eta,-\de})) = \dim(\ker(P_{F_0,-\de}))$. From
  Proposition 1.12 in \cite{Bartnik1986}\footnote{The proof of Proposition 1.12 in \cite{Bartnik1986} only relies on estimate (1.26) there, which corresponds to \eqref{estimation elliptique} in our article. While Proposition 1.12 in \cite{Bartnik1986} is stated for scalar operators, note that it extends immediately to principally scalar operators such as the ones we consider.} there
  exists $0<\eta'(F_0)\leq \eta(F_0)$ and $C(F_0)>0$ such that if
  $\eta\leq \eta'(F_0)$ then
  \begin{align*}
    \l \mathbb{U} - \ker(P_{F_\eta,-\de}) \r_{H^2_{-\de}(\Si)} \leq C(F_0)\pth{ \l P(\mathbb{U})\r_{L^2_{-\de-2}(\Si)} + \l (B_\nu+F_\eta)(\mathbb{U})\r_{H^\half(\dr\Si)}},
  \end{align*}
  for all $\mathbb{U}\in H^2_{-\de}(\Sigma)$. We apply this to
  $\mathbb{V}_0\in \ker(P_{F_0,-\de})$ as chosen above. Since
  $\ker(P_{F_\eta,-\de})$ is finite-dimensional, there exists
  some $\mathbb{U}_\eta\in \ker(P_{F_\eta,-\de})$ such that
  \begin{align}
    \label{bartnik 1}
    \l \mathbb{V}_0 - \mathbb{U}_\eta \r_{H^2_{-\de}(\Si)}
    \lesssim C(F_0)  \l \mathbb{V}_0\r_{H^2_{-\de}(\Si)}\eta,
  \end{align}
  where we have used the continuous embeddings
  $H^2_{-\de}(\Si)\xhookrightarrow{}H^{\frac{3}{2}}(\dr\Si)\xhookrightarrow{}L^\infty(\dr\Si)$. Then
  using \eqref{IPP}, the fact that
  $\mathbb{V}_0^*\in \ker(P^*_{F_0,-1+\de})$, and that
  $\mathbb{U}_\eta\in \ker(P_{F_\eta,-\de})$, we compute
  \begin{align*}
     \frac{\l H_0 \r_{H^{\frac{3}{2}}(\dr\Si)}}{\eta}\int_\Si\pth{P(\mathbb{U}_\eta)\cdot \mathbb{V}_0^* - \mathbb{U}_\eta\cdot P^*(\mathbb{V}_0^*) }  =  \int_{\dr\Si} H_0 (\mathbb{U}_\eta)\cdot \mathbb{V}_0^* .
  \end{align*}
  Since $P(\mathbb{U}_\eta)=P^*(\mathbb{V}_0^*)=0$ we have proved that
  \begin{align*}
      \int_{\dr\Si} H_0 (\mathbb{U}_\eta)\cdot \mathbb{V}_0^* = 0.
  \end{align*}
  However \eqref{def h zero} and \eqref{bartnik 1} imply
  \begin{align*}
    \left| \int_{\dr\Si} H_0 (\mathbb{U}_\eta)\cdot \mathbb{V}_0^* - 1 \right| & \leq C C(F_0) \l H_0\r_{H^{\frac{3}{2}}(\dr\Si)} \l \mathbb{V}_0^* \r_{H^2_{-\de}(\Si)}   \l \mathbb{V}_0\r_{H^2_{-\de}(\Si)}\eta,
  \end{align*}
for some universal constant $C>0$. Choosing
  \begin{align*}
      \eta_1(F_0)\vcentcolon = \min\pth{ \eta'(F_0), \frac{1}{2C C(F_0) \l H_0\r_{H^{\frac{3}{2}}(\dr\Si)} \l \mathbb{V}_0^* \r_{H^2_{-\de}(\Si)}   \l \mathbb{V}_0\r_{H^2_{-\de}(\Si)}} }
  \end{align*}
  and taking $0<\eta\leq \eta_1(F_0)$ implies a contradiction. This
  proves that $\dim(\ker(P_{F_\eta,-\de})) < \dim(\ker(P_{F_0,-\de}))$
  and concludes the proof of the lemma, after setting
  $F_1\vcentcolon=F_\eta$.
\end{proof}

We can iterate the result of Lemma \ref{lem bartnik 2} to
obtain the following corollary.
\begin{corollary}
  \label{coro BC}
  Let $\de\in(0,1)$. Let $\mathcal{U}$ and $\mathcal{V}$ two open subsets of $\dr\Si$ with $\overline{\mathcal{U}}\cap \overline{\mathcal{V}}=\emptyset$. 
  There exists $F_\de\in \EE(\dr\Si)$ such that $F_{\de|_{\mathcal{V}}}=\mathrm{Id}$, $F_{\de|_{\mathcal{U}}}=0$ and
  \begin{align}
      \ker(P_{F_\de,-\de}) & =\{0\},\label{f de 1}
      \\ \ker(P^*_{F_\de,-1+\de}) & =\{0\}.\label{f de 2}
  \end{align}
\end{corollary}

\begin{proof}
First, note that \eqref{f de 1} implies \eqref{f de 2}. Indeed, from \eqref{kernel P^*}, we have that
  $\dim\pth{\ker\pth{P^*_{F_{\de},-1+\de}}}=\dim\pth{\ker\pth{\pth{P_{F_{\de},-\de}}^*}}$.
  In view of the fact that $P_{F_{\de},-\de}$ has Fredholm index zero
  from Proposition \ref{prop fredholm} this becomes
  \begin{align*}
      \dim\pth{\ker\pth{P^*_{F_{\de},-1+\de}}} = \dim\pth{\ker\pth{P_{F_{\de},-\de}}}.
  \end{align*}
Therefore, it only remains to construct $F_\de$ satisfying the requirements of the corollary and such that \eqref{f de 1} holds. Since $\overline{\mathcal{U}}\cap \overline{\mathcal{V}}=\emptyset$, we may consider $\chi_{b}$ a smooth function on $\dr\Si$ such that $\chi_{b|_{\mathcal{U}}}=0$ and $\chi_{b|_{\mathcal{V}}}=1$ and set $F_0\vcentcolon =\chi_{b}\mathrm{Id}$. Let $n_0\vcentcolon = \dim(\ker(P_{F_0,-\de}))$. If $n_0=0$, the lemma is proved. If $n_0>0$, we construct iteratively a finite family of
  matrices $(F_0,F_1,\dots,F_{n_*})$ in the following way:
  \begin{enumerate}
  \item We apply Lemma \ref{lem bartnik 2} to $F_0$ and $\mathcal{W}\vcentcolon =\mathcal{U}\cup \mathcal{V}$ (since $\overline{\mathcal{U}}\cap \overline{\mathcal{V}}=\emptyset$, $\dr\Si\setminus \mathcal{W}$ does indeed contain an open set), which
    gives the existence of $F_1\in \EE(\dr\Si)$ with $\dim(\ker(P_{F_1,-\de})) < n_0$ and
    $(F_1-F_0)_{|_{\mathcal{W}}}=0$.
  \item If $(F_0,F_1,\dots,F_j)$ is constructed, then we distinguish two cases:
    \begin{enumerate}
    \item If $\dim(\ker(P_{F_j,-\de}))=0$, then set $n_*=j$ and the
      iteration ends.
    \item Otherwise, $\dim(\ker(P_{F_j,-\de}))>0$, and Lemma \ref{lem
        bartnik 2} again implies the existence of $F_{j+1}\in \EE(\dr\Si)$ such that
      $\dim(\ker(P_{F_{j+1},-\de}))<\dim(\ker(P_{F_j,-\de}))$ and
      $(F_{j+1}-F_j)_{|_{\mathcal{W}}}=0$.      
    \end{enumerate}
  \end{enumerate}
  Since $\pth{\dim(\ker(P_{F_j,-\de}))}_{j}$ is a decreasing sequence of integers and since $\dim(\ker(P_{F_0,-\de}))=n_0$, we have that
  $n_{*}\le n_0$. We thus have a finite family
  $(F_0,F_1,\dots,F_{n_*})$ in $\EE(\dr\Si)$ with
  $\ker(P_{F_{n_*},-\de})=\{0\}$. Moreover, we can prove by induction that $(F_{n_*}-F_0)_{|_{\mathcal{U}\cup \mathcal{V}}}=0$. Since $F_{0|_{\mathcal{V}}}=\mathrm{Id}$ and $F_{0|_{\mathcal{U}}}=0$, this proves that $F_\de\vcentcolon = F_{n_*}$ does satisfy the requirements of the lemma.
\end{proof}

\subsection{Adapted harmonic polynomials}

We recall the definition of harmonic polynomials and vectorial
harmonic polynomials in Euclidean space.

\begin{definition}
  \label{def We}
  We identify the global coordinates of $\RRR^3$ and $\Si$ and consider the
  real spherical harmonics $Y_{j,\ell}$ for $j\geq 0$ and
  $-j\leq \ell \leq j$ with an arbitrary normalization.
  \begin{enumerate}[label=(\roman*)]
  \item For $j\geq 1$ and $-(j-1)\leq \ell \leq j-1$, we define scalar functions
    \begin{align*}
      w_{j,\ell} & \vcentcolon = r^{j-1} Y_{j-1,\ell},
    \end{align*}
    and if in addition $k=1,2,3$, we define vector fields
    \begin{align*}
      W_{j,\ell,k} & \vcentcolon = w_{j,\ell}\dr_k.
    \end{align*}
  \item For $j\geq 1$, $-j\leq \ell \leq j$ and $\a=0,1,2,3$, we define
    pairs consisting of a scalar function and a vector field
    \begin{equation*}
      \mathbb{W}_{j,\ell,\a}^{(e)} \vcentcolon = 
      \left\{
        \begin{aligned}
          &(w_{j,\ell},0), \quad \text{if $\a=0$,}
          \\& (0, W_{j,\ell,\a}), \quad \text{if $\a=1,2,3$,}
        \end{aligned}
      \right.
    \end{equation*}
  \end{enumerate}

\end{definition}

\begin{remark}
  \label{remark basis harmonic polynomials}
  Note that for each $q\in\mathbb{N}^*$ the family 
  \begin{align*}
    \pth{\mathbb{W}^{(e)}_{j,\ell,\a}}_{1\leq j \leq q, -(j-1)\leq \ell\leq j-1,\a=0,1,2,3}
  \end{align*}
  is a basis of homogeneous polynomials of degree $j-1$ with $0\leq j-1\leq q-1$
  moreover satisfying $\De_e\pth{\mathbb{W}^{(e)}_{j,\ell,\a}}=0$ where $\De_e$
  denotes here the Euclidean Laplacian.
\end{remark}

Using the harmonic polynomials of \Cref{def We}, we construct a family of harmonic polynomials adapted to $P_{F,\delta}$. From now on, $F_\de$ (for $\de\in(0,1)$) will always refer to the one constructed in Corollary \ref{coro BC}.
\begin{lemma}
  \label{lem harmonic polynomials}
  Let $j\geq 1$, $-(j-1)\leq \ell \leq j-1$, $\a=0,1,2,3$ and $\de\in(0,1)$. There
  exists $\widetilde{\mathbb{W}}_{j,\ell,\a}\in H^2_{j-2+\de}(\Si)$
  such that
  \begin{align*}
    \mathbb{W}_{j,\ell,\a} & \vcentcolon = \mathbb{W}^{(e)}_{j,\ell,\a} + \widetilde{\mathbb{W}}_{j,\ell,\a}
  \end{align*}
  satisfies
  \begin{equation*}
    \left\{
      \begin{aligned}
        P^*\pth{\mathbb{W}_{j,\ell,\a}} & = 0,\quad \text{on $\Si$,}
        \\ (B_\nu^*+F_\de^*)\pth{\mathbb{W}_{j,\ell,\a}} & = 0,\quad \text{on $\dr\Si$.}
      \end{aligned}
    \right.
  \end{equation*}
\end{lemma}

\begin{proof}
  It suffices to solve for
  $\widetilde{\mathbb{W}}_{j,\ell,\a}\in H^2_{j-2+\delta}(\Sigma)$
  such that
  \begin{equation*}
    \left\{
      \begin{aligned}
        P^*\pth{\widetilde{\mathbb{W}}_{j,\ell,\a}} & = -P^*\pth{\mathbb{W}^{(e)}_{j,\ell,\a}},\quad \text{on $\Si$,}
        \\ (B_\nu^*+F_\de^*)\pth{\widetilde{\mathbb{W}}_{j,\ell,\a}} & = -(B_\nu^*+F_\de^*)\pth{\mathbb{W}^{(e)}_{j,\ell,\a}},\quad \text{on $\dr\Si$.}
      \end{aligned}
    \right.
  \end{equation*}
  Now, observe that $P^{*}$ is asymptotic to the Euclidean Laplacian in view of Definition \ref{def AF}. As a result, since
  $\mathbb{W}_{j,\ell,\alpha}^{(e)}\in H^4_{j-1+\delta}(\Sigma)$, we have that
  \begin{equation*}
    P^{*}\pth{\mathbb{W}^{(e)}_{j,\ell,\a}} \in H^2_{j-4+\de}(\Sigma),
  \end{equation*}   
  since $\De_e\pth{\mathbb{W}^{(e)}_{j,\ell,\a}}=0$ in view of \Cref{remark basis harmonic polynomials}. We can thus rewrite the above system for
  $\widetilde{\mathbb{W}}_{j,\ell,\a}$ as
  \begin{align*}
    P_{F_\de,j-2+\de}^*\pth{\widetilde{\mathbb{W}}_{j,\ell,\a}}=-\pth{ P^*\pth{\mathbb{W}^{(e)}_{j,\ell,\a}},(B^*_\nu+F_\de^*)\pth{\mathbb{W}^{(e)}_{j,\ell,\a}} }.
  \end{align*}
  Since $j\geq 1$ we have
  $\ker(P_{F_\de,1-j-\de})\subset \ker(P_{F_\de,-\de})=\{0\}$, where
  we used \eqref{f de 1}. According to \eqref{kernel P^**} this
  implies $\ker\pth{\pth{P^*_{F_\de,j-2+\de}}^*}=\{0\}$, and hence that
  $P^*_{F_\de,j-2+\de}$ is surjective (since $P^*_{F_\de,j-2+\de}$ is semi-Fredholm thanks to Proposition \ref{prop fredholm}). Therefore we can always solve
  for $\widetilde{\mathbb{W}}_{j,\ell,\a}$.
\end{proof}

Using the adapted harmonic polynomials constructed in \Cref{lem harmonic polynomials}, we are now able to specify the orthogonality conditions necessary to invert $P_{F_\de,-q-\delta}$. 

\begin{lemma}
  \label{lem image}
  Let $q\in\mathbb{N}^*$, $0<\de<1$ and let
  $\mathbb{Y}\in L^2_{-q-\de-2}(\Si)$. Then
  $(\mathbb{Y},0)\in \mathrm{im}(P_{F_\de,-q-\de})$ if and only if
  \begin{align*}
    \ps{\mathbb{W}_{j,\ell,\a},\mathbb{Y}}_{L^2(\Si)}=0 , \quad \text{for all $1\leq j\leq q$, $-(j-1)\leq \ell \leq j-1$ and $\a=0,1,2,3$}.
  \end{align*}
\end{lemma}

\begin{remark}\label{remark 1.17}
In the proof of this lemma below, we will apply Theorem 1.17 from \cite{Bartnik1986}. With the notations of this paper we have $k^{-}(p+\de)=p$ for any $p\in\mathbb{N}$ and $\de\in(0,1)$ since in dimension $3$ the set of so-called exceptional values is $\mathbb{Z}$.
\end{remark}

\begin{proof}
  We start by proving that
  \begin{align}\label{description noyau}
    \ker\pth{P^*_{F_\de,q-1+\de}}=\mathrm{span}\big\{\mathbb{W}_{j,\ell,\a}, \text{$1\leq j\leq q$, $-(j-1)\leq \ell \leq j-1$ and $\a=0,1,2,3$}\big\},
  \end{align}
  for all $q\geq 1$. The inclusion from right to left is immediate by
  definition of $\mathbb{W}_{j,\ell,\a}$ in the Lemma \ref{lem
    harmonic polynomials} and the fact that $j\leq q$. For the
  inclusion from left to right we proceed by induction on $q\geq 1$:
\begin{itemize}
\item For $q=1$, let $\mathbb{W}\in
    \ker\pth{P^*_{F_\de,\de}}$. Since $P^*$ is asymptotic to the Euclidean Laplacian $\De$,
    Theorem 1.17 in \cite{Bartnik1986},
    Remarks \ref{remark basis harmonic polynomials} and \ref{remark 1.17} imply that there
    exists scalars $\la^\a$ such that
    $\mathbb{W}-\la^\a \mathbb{W}^{(e)}_{1,0,\a}\in
    H^2_{-1+\de}(\Si)$. Lemma \ref{lem harmonic polynomials} thus implies
    that
    $\mathbb{W}-\la^\a \mathbb{W}_{1,0,\a}\in
    \ker\pth{P^*_{F_\de,-1+\de}}$. Thanks to \eqref{f de 2}, we obtain
    $\mathbb{W}=\la^\a \mathbb{W}_{1,0,\a}$. This proves
    \eqref{description noyau} for $q=1$.
\item Assume that we have
    \begin{equation}
    \label{induction}
     \hspace{1cm} \ker\pth{P^*_{F_\de,q-1+\de}}=\mathrm{span}\big\{\mathbb{W}_{j,\ell,\a}, \text{$1\leq j\leq q$, $-(j-1)\leq \ell \leq j-1$ and $\a=0,1,2,3$}\big\}
    \end{equation}
    for some $q\geq 1$ and let
    $\mathbb{W}\in \ker\pth{P^*_{F_\de,q+\de}}$. From Theorem 1.17 in
    \cite{Bartnik1986}, Remarks \ref{remark
      basis harmonic polynomials} and \ref{remark 1.17} we again obtain the existence of
    scalars $\la^{\ell,\a}$ such that
    $\mathbb{W}-\la^{\ell,\a}\mathbb{W}^{(e)}_{q+1,\ell,\a}\in
    H^2_{q-1+\de}(\Si)$. Lemma \ref{lem harmonic polynomials} thus
    implies that
    $\mathbb{W}-\la^{\ell,\a}\mathbb{W}_{q+1,\ell,\a}\in
    \ker\pth{P^*_{F_\de,q-1+\de}}$. From \eqref{induction} this gives
    \begin{align*}
      \mathbb{W}-\la^{\ell,\a}\mathbb{W}_{q+1,\ell,\a}\in \mathrm{span}\big\{\mathbb{W}_{j,\ell,\a}, \text{$1\leq j\leq q$, $-(j-1)\leq \ell \leq j-1$ and $\a=0,1,2,3$}\big\}
    \end{align*}
    and thus 
    \begin{align*}
      \mathbb{W} \in \mathrm{span}\big\{\mathbb{W}_{j,\ell,\a}, \text{$1\leq j\leq q+1$, $-(j-1)\leq \ell \leq j-1$ and $\a=0,1,2,3$}\big\}.
    \end{align*}
\end{itemize}
    This concludes the proof of \eqref{description noyau}. Now, by
    definition of the adjoint and the semi-Fredholm property of $P_{F_\de,-q-\de}$ (see Proposition \ref{prop fredholm}),
    $(\mathbb{Y},0)\in \mathrm{im}(P_{F_\de,-q-\de})$ if and only if
    $\ps{\mathbb{Y},\mathbb{V}}_{L^2(\Si)}=0$ for all $\mathbb{V}$
    such that $(\mathbb{V},\mathbbm{h})\in \ker((P_{F_\de,-q-\de})^*)$
    for some $\mathbbm{h}$. Thanks to \eqref{kernel P^*} we obtain
  \begin{align*}
    (\mathbb{Y},0)\in \mathrm{im}(P_{F_\de,-q-\de}) &\iff \ps{\mathbb{Y},\mathbb{V}}_{L^2(\Si)}=0, \quad \text{for all $\mathbb{V}\in \ker\pth{P^*_{F_\de,q-1+\de}}$}.
  \end{align*}
  Using \eqref{description noyau} concludes the proof of \Cref{lem image}.
\end{proof}

\subsection{Conclusion}

We summarize the properties of the operator $P$ in \eqref{def P} proved so far.
\begin{proposition}
  \label{prop P final}
    Let $q\in\mathbb{N}^*$ and $0<\de<1$. There exists $F_\de\in \EE(\dr\Si)$ and a family 
    \begin{equation}
      \label{family W}
        (\mathbb{W}_{j,\ell,\a})_{1\leq j\leq q, -(j-1)\leq \ell \leq j-1,\a=0,1,2,3}    
    \end{equation}
    with the following properties
    \begin{enumerate}[label=(\roman*)]
    \item there exists $\mathcal{U}$ and $\mathcal{V}$ open subsets of $\dr\Si$ such that $F_{\de|_{\mathcal{V}}}=\mathrm{Id}$ and $F_{\de|_{\mathcal{U}}}=0$,
    \item the family \eqref{family W} is linearly independent and each
      $\mathbb{W}_{j,\ell,\a}$ satisfies the estimate
      $\left| \mathbb{W}_{j,\ell,\a} \right|\lesssim r^{j-1}$ and the
      system
      \begin{equation*}
        \left\{
          \begin{aligned}
            P^*\pth{\mathbb{W}_{j,\ell,\a}} & = 0,\quad \text{on $\Si$,}
            \\ (B_\nu^*+F_\de^*)\pth{\mathbb{W}_{j,\ell,\a}} & = 0,\quad \text{on $\dr\Si$,}
          \end{aligned}
        \right.
      \end{equation*}
    \item if $\mathbb{Y}\in L^2_{-q-\de-2}(\Si)$ satisfies
      \begin{align}\label{condition Y}
        \ps{\mathbb{W}_{j,\ell,\a},\mathbb{Y}}_{L^2(\Si)}=0 , \quad \text{for all $1\leq j\leq q$, $-(j-1)\leq \ell \leq j-1$ and $\a=0,1,2,3$},
      \end{align}
      then there exists a unique $\mathbb{U}\in H^2_{-q-\de}(\Si)$
      with
      $\l \mathbb{U}\r_{H^2_{-q-\de}(\Si)}\lesssim \l
      \mathbb{Y}\r_{L^2_{-q-\de-2}(\Si)}$ such that
      \begin{equation*}
        \left\{
          \begin{aligned}
            P\pth{\mathbb{U}} & = \mathbb{Y},\quad \text{on $\Si$,}
            \\ (B_\nu+F_\de)\pth{\mathbb{U}} & = 0,\quad \text{on $\dr\Si$.}
          \end{aligned}
        \right.
      \end{equation*}
    \end{enumerate}
\end{proposition}

\begin{proof}
  We consider $F_\de$ as constructed in
  Corollary \ref{coro BC}. The family \eqref{family W} has been
  constructed in Lemma \ref{lem harmonic polynomials} and its linear
  independence follows from the linear independence of the family
  \begin{align*}
    \pth{\mathbb{W}^{(e)}_{j,\ell,\a}}_{1\leq j\leq q, -(j-1)\leq \ell \leq j-1,\a=0,1,2,3} 
  \end{align*}
  and the asymptotics
  $\mathbb{W}_{j,\ell,\a}-\mathbb{W}^{(e)}_{j,\ell,\a}=\GO{r^{j-2+\de}}$
  at infinity. If $\mathbb{Y}\in L^2_{-q-\de-2}(\Si)$ satisfies
  \eqref{condition Y} then Lemma \ref{lem image} implies that
  $(\mathbb{Y},0)\in \mathrm{im}(P_{F_\de,-q-\de})$, so that there
  exists $\mathbb{U}\in H^2_{-q-\de}(\Si)$ solving
  $P_{F_\de,-q-\de}(\mathbb{U})=(\mathbb{Y},0)$. Uniqueness follows
  from the choice of $F_\de$ (Corollary \ref{coro BC}) which implies
  that $\ker(P_{F_\de,-q-\de})\subset \ker(P_{F_\de,-\de})=\{0\}$. Finally, the
  estimate
  $\l \mathbb{U}\r_{H^2_{-q-\de}(\Si)}\lesssim \l
  \mathbb{Y}\r_{L^2_{-q-\de-2}(\Si)}$ follows from (1.31) in
  \cite{Bartnik1986}.
\end{proof}

\section{Proof of the main theorem}
\label{section proof main result}

This section is devoted to the proof of Theorem \ref{theo general}. 

\subsection{Conformal formulation of the constraint equations}
\label{section conformal formulation}

Let $(\Si,\bar{g},\bar{\pi})$ an asymptotically flat initial data set satisfying the regularity assumptions of Theorem \ref{theo general}, and let $(\wc{g},\wc{\pi})\in H^2_{-q-\de}(\Si)\times H^1_{-q-\de-1}(\Si)$ for some $q\in\mathbb{N}^*$ and $0<\de<1$. For $\p=(\breve{g},\breve{\pi})$ compactly supported we define 
\begin{equation}
  \label{def g hat pi hat}
  (\hat{g},\hat{\pi})(\p)  \vcentcolon = (\bar{g} + \wc{g} , \bar{\pi} + \wc{\pi}) + \p.
\end{equation}
Following \cite{Corvino2006}, we look for
solutions of the constraint equations under the form
\begin{equation}
  \label{ansatz}
  (g,\pi) = \pth{ (1+\wc{u})^4 \hat{g}(\p), (1+\wc{u})^2\pth{ \hat{\pi}(\p) + L_{\hat{g}(\p)}\wc{X}}}.
\end{equation}

\begin{lemma}
  \label{lem conformal formulation}
  If $(g,\pi)$ is defined by \eqref{ansatz} then the constraint
  equations $\Phi(g,\pi)=0$ rewrite as
  \begin{align}\label{eq u X}
    P\pth{\wc{\mathbb{U}}} & = D\Phi[\bar{g},\bar{\pi}]((\wc{g},\wc{\pi})+\p) + \RR\pth{\p,\wc{\mathbb{U}}}
  \end{align}
  where the operator $P$ is defined in \eqref{def P},
  $\wc{\mathbb{U}}=\pth{\wc{u},\wc{X}}$ and the remainder is defined
  by
  \begin{equation*}
    \begin{aligned}
      \RR\pth{\p,\wc{\mathbb{U}}} & \vcentcolon = \bigg(- Q_\HH[\bar{g},\bar{\pi}](\wc{g} + \breve{g}, \wc{\pi} + \breve{\pi})
      \\&\hspace{1cm} -\wc{u}\pth{D\HH[\bar{g},\bar{\pi}](\wc{g} + \breve{g}, \wc{\pi} + \breve{\pi}) + Q_\HH[\bar{g},\bar{\pi}](\wc{g} + \breve{g}, \wc{\pi} + \breve{\pi})}
      \\&\hspace{1cm} - (1+\wc{u}) Q_\HH[\hat{g}(\p),\hat{\pi}(\p)]\pth{0,L_{\hat{g}(\p)}\wc{X}} - \wc{u} \pth{  - 2\hat{\pi}(\p)^{ij} \LL_{\wc{X}} \hat{g}(\p)_{ij} + \tr_{\hat{g}(\p)}\hat{\pi}(\p) \div_{\hat{g}(\p)}\wc{X} }
      \\&\hspace{1cm} + 8\pth{\De_{\hat{g}(\p)}-\De_{\bar{g}}}\wc{u}   + 2\pth{\hat{\pi}(\p)^{ij} \LL_{\wc{X}} \hat{g}(\p)_{ij}- \bar{\pi}^{ij} \LL_{\wc{X}} \bar{g}_{ij}} - \pth{\tr_{\hat{g}(\p)}\hat{\pi}(\p) \div_{\hat{g}(\p)}-\tr_{\bar{g}}\bar{\pi} \div_{\bar{g}}}\wc{X},
      \\&\hspace{1cm} 2 Q_\MM[\bar{g},\bar{\pi}](\wc{g} + \breve{g}, \wc{\pi} + \breve{\pi})_i
      \\&\hspace{1cm} +4(1+\wc{u})^{-1} \hat{g}(\p)^{j\ell} \pth{2 L_{\hat{g}(\p)}\wc{X}_{i\ell} \dr_j\wc{u} - L_{\hat{g}(\p)}\wc{X}_{j\ell} \dr_i \wc{u}}
      \\&\hspace{1cm} + 4 ((1+\wc{u})^{-1}-1) \hat{g}(\p)^{j\ell} \pth{2  \hat{\pi}(\p)_{i\ell} \dr_j\wc{u}  - \hat{\pi}(\p)_{j\ell}  \dr_i \wc{u}}
      \\&\hspace{1cm} + 2\pth{\div_{\hat{g}(\p)}L_{\hat{g}(\p)}-\div_{\bar{g}}L_{\bar{g}}}\wc{X}_i  + 8 \pth{ \hat{g}(\p)^{j\ell} \hat{\pi}(\p)_{i\ell}   -   \bar{g}^{j\ell} \bar{\pi}_{i\ell}} \dr_j\wc{u}   - 4 \pth{  \tr_{\hat{g}(\p)}\hat{\pi}(\p) -  \tr_{\bar{g}}\bar{\pi}}  \dr_i \wc{u} \bigg),
    \end{aligned}
  \end{equation*}
  where $Q_\HH[\bar{g},\bar{\pi}]$ and $Q_\MM[\bar{g},\bar{\pi}]$ are defined in \eqref{expansion H}-\eqref{expansion M}.
\end{lemma}

\begin{proof}
  In the proof of Theorem 1 in
  \cite{Corvino2006} we find the following
  formulas:
  \begin{align*}
    u^5\HH\pth{ u^4 \hat{g}(\p), u^2\pth{ \hat{\pi}(\p) + L_{\hat{g}(\p)}X}} & =  -8\De_{\hat{g}(\p)} u + u \HH\pth{\hat{g}(\p),\hat{\pi}(\p) + L_{\hat{g}(\p)}X } 
  \end{align*}
  and
  \begin{align*}
    &u^2\MM\pth{ u^4 \hat{g}(\p), u^2\pth{ \hat{\pi}(\p) + L_{\hat{g}(\p)}X}}_i 
    \\&\qquad  =   \div_{\hat{g}(\p)}\pth{\hat{\pi}(\p) + L_{\hat{g}(\p)}X}_i + 2u^{-1} \hat{g}(\p)^{j\ell} \pth{2  \pth{\hat{\pi}(\p) + L_{\hat{g}(\p)}X}_{i\ell} \dr_ju -  \pth{\hat{\pi}(\p) + L_{\hat{g}(\p)}X}_{j\ell} \dr_i u } . 
  \end{align*}
  On the one hand, thanks to \eqref{expansion H} and \eqref{DH}, we can compute
  $D\HH[\hat{g}(\p),\hat{\pi}(\p)](0,L_{\hat{g}}X)$ and rewrite the
  equation
  \begin{align*}
    \HH\pth{ u^4 \hat{g}(\p), u^2\pth{ \hat{\pi}(\p) + L_{\hat{g}(\p)}X}} = 0
  \end{align*}
  as
  \begin{align*}
    -8\De_{\hat{g}(\p)} u   - 2\hat{\pi}(\p)^{ij} \LL_X \hat{g}(\p)_{ij}& + \tr_{\hat{g}(\p)}\hat{\pi}(\p) \div_{\hat{g}(\p)}X  
    \\& = - \HH(\hat{g}(\p),\hat{\pi}(\p)) 
    \\&\quad -(u-1)\HH(\hat{g}(\p),\hat{\pi}(\p)) - u Q_\HH[\hat{g}(\p),\hat{\pi}(\p)](0,L_{\hat{g}(\p)}X)
    \\&\quad - (u-1) \pth{  - 2\hat{\pi}(\p)^{ij} \LL_X \hat{g}(\p)_{ij} + \tr_{\hat{g}(\p)}\hat{\pi}(\p) \div_{\hat{g}(\p)}X }.
  \end{align*}
  Using now the definition \eqref{def g hat pi hat}, \eqref{expansion H} and $\HH(\bar{g},\bar{\pi})=0$, we rewrite this equation as
  \begin{align*}
    -8\De_{\bar{g}} u   - 2\bar{\pi}^{ij} \LL_X \bar{g}_{ij} + \tr_{\bar{g}}\bar{\pi} \div_{\bar{g}}X   & = - D\HH[\bar{g},\bar{\pi}](\wc{g} + \breve{g}, \wc{\pi} + \breve{\pi}) 
    \\&\quad - Q_\HH[\bar{g},\bar{\pi}](\wc{g} + \breve{g}, \wc{\pi} + \breve{\pi})
    \\&\quad -(u-1)\pth{D\HH[\bar{g},\bar{\pi}](\wc{g} + \breve{g}, \wc{\pi} + \breve{\pi}) + Q_\HH[\bar{g},\bar{\pi}](\wc{g} + \breve{g}, \wc{\pi} + \breve{\pi})}
    \\&\quad - u Q_\HH[\hat{g}(\p),\hat{\pi}(\p)](0,L_{\hat{g}(\p)}X)
    \\&\quad - (u-1) \pth{  - 2\hat{\pi}(\p)^{ij} \LL_X \hat{g}(\p)_{ij} + \tr_{\hat{g}(\p)}\hat{\pi}(\p) \div_{\hat{g}(\p)}X }
    \\&\quad + 8\De_{\hat{g}(\p)} u   + 2\hat{\pi}(\p)^{ij} \LL_X \hat{g}(\p)_{ij} - \tr_{\hat{g}(\p)}\hat{\pi}(\p) \div_{\hat{g}(\p)}X 
    \\&\quad -8\De_{\bar{g}} u   - 2\bar{\pi}^{ij} \LL_X \bar{g}_{ij} + \tr_{\bar{g}}\bar{\pi} \div_{\bar{g}}X. 
  \end{align*}
  On the other hand, thanks to \eqref{expansion M} and \eqref{DM}, as well as \eqref{def g hat pi hat} and $\div_{\bar{g}}\bar{\pi}=0$, we rewrite the equation
  \begin{align*}
    \MM\pth{ u^4 \hat{g}(\p), u^2\pth{ \hat{\pi}(\p) + L_{\hat{g}(\p)}X}}_i = 0
  \end{align*}
  as
  \begin{align*}
    -2\div_{\bar{g}}L_{\bar{g}}X_i  - 8  \bar{g}^{j\ell} \bar{\pi}_{i\ell} \dr_ju   + 4   \tr_{\bar{g}}\bar{\pi}  \dr_i u     & =  2D\MM[\bar{g},\bar{\pi}](\wc{g} + \breve{g}, \wc{\pi} + \breve{\pi})_i +2 Q_\MM[\bar{g},\bar{\pi}](\wc{g} + \breve{g}, \wc{\pi} + \breve{\pi})_i
    \\&\quad + 8 u^{-1} \hat{g}(\p)^{j\ell}L_{\hat{g}(\p)}X_{i\ell} \dr_ju - 4 u^{-1}  \tr_{\hat{g}(\p)} L_{\hat{g}(\p)}X \dr_i u
    \\&\quad + 8 (u^{-1}-1) \hat{g}(\p)^{j\ell} \hat{\pi}(\p)_{i\ell} \dr_ju  - 4 (u^{-1}-1)  \tr_{\hat{g}(\p)}\hat{\pi}(\p)  \dr_i u 
    \\&\quad + 2\div_{\hat{g}(\p)}L_{\hat{g}(\p)}X_i  + 8  \hat{g}(\p)^{j\ell} \hat{\pi}(\p)_{i\ell} \dr_ju   - 4   \tr_{\hat{g}(\p)}\hat{\pi}(\p)  \dr_i u 
    \\&\quad -2\div_{\bar{g}}L_{\bar{g}}X_i  - 8  \bar{g}^{j\ell} \bar{\pi}_{i\ell} \dr_ju   + 4   \tr_{\bar{g}}\bar{\pi}  \dr_i u . 
  \end{align*}
  Applying the above formula to $u=1+\wc{u}$ and $X=\wc{X}$ concludes
  the proof of the lemma.
\end{proof}

The proof of Theorem \ref{theo general} reduces to whether we can
construct $\p$ such that \eqref{eq u X} admits a solution
$\wc{\mathbb{U}}\in H^2_{-q-\de}(\Si)$. Thanks to Proposition
\ref{prop P final}, we need $\p$ to ensure orthogonality conditions
with the $\mathbb{W}_{j,\ell,\a}$ for $1\leq j \leq q$. Because of the
remainder term in \eqref{eq u X}, the definition of $\p$ is non-linear
and will follow from a fixed point argument performed in Section
\ref{section corrector}. More precisely for each
$\wc{\mathbb{U}}\in H^2_{-q-\de}(\Si)$ small enough we will construct
and control $\p\pth{\wc{\mathbb{U}}}$ so that the right-hand side of
\eqref{eq u X} is orthogonal to $\mathbb{W}_{j,\ell,\a}$ for
$1\leq j \leq q$. We then solve the elliptic system in Section
\ref{section solving system} and thus conclude the proof of Theorem \ref{theo general}.

\subsection{Estimating the remainder}
\label{section remainder}
We start by estimating the remainder term in \eqref{eq u X}, i.e.
$\RR\pth{\p,\wc{\mathbb{U}}}$. For the sake of clarity, we introduce
the schematic notations
\begin{align*}
    \bar{\Ga} = \{ \nab\bar{g},\bar{\pi}\}, \quad \wc{\Ga} = \{ \nab\wc{g},\wc{\pi}\}, \quad \breve{\Ga} = \{ \nab\breve{g},\breve{\pi}\}, \quad \hat{\Ga}(\p) = \{ \nab\hat{g}(\p),\hat{\pi}(\p)\}.
\end{align*}
Since $(\wc{g},\wc{\pi})$, $\p$ and $\wc{\mathbb{U}}$ will be small in appropriate spaces, we only write down quadratic terms
when manipulating $\RR\pth{\p,\wc{\mathbb{U}}}$ since cubic and higher order
terms are smaller and enjoy more decay and thus are easier
to treat. Note also that, from now on, $\wc{g}$ and $\wc{\pi}$ are fixed and are assumed to satisfy \eqref{assumption check} for some $\e>0$ that will be fixed in the proofs of Lemmas \ref{lem corrector map} and \ref{dernier lemme}.

Thanks to its expression given in Lemma \ref{lem conformal
  formulation} and the various schematic expressions of Lemma \ref{lem
  linearized constraint} we find
\begin{align*}
    \RR\pth{\p,\wc{\mathbb{U}}} & = \RR_1\pth{\p} + \RR_2\pth{\p,\wc{\mathbb{U}}} + \RR_3\pth{\p,\wc{\mathbb{U}}} + \text{higher-order terms}
\end{align*}
where we have schematically
\begin{align*}
    \RR_1\pth{\p} & = (\wc{g}+\breve{g})\nab\pth{\wc{\Ga}+\breve{\Ga}} + \bar{\Ga}(\wc{g}+\breve{g})\pth{\wc{\Ga}+\breve{\Ga}} +  \pth{\wc{\Ga}+\breve{\Ga}}^2 + \bar{\Ga}^2(\wc{g}+\breve{g})^2,
    \\ \RR_2\pth{\p,\wc{\mathbb{U}}} & = \hat{g}(\p) \pth{\nab \wc{\mathbb{U}}}^2  + \hat{\Ga}(\p)  \wc{\mathbb{U}} \pth{ \nab \wc{\mathbb{U}} + \hat{\Ga}(\p) \wc{\mathbb{U}}},
    \\ \RR_3\pth{\p,\wc{\mathbb{U}}} & = (\wc{g}+\breve{g})\nab^2\wc{\mathbb{U}} + \pth{ (\wc{g}+\breve{g})\bar{\Ga} + \wc{\Ga}+\breve{\Ga}}\nab \wc{\mathbb{U}} 
    \\&\quad + \pth{ \nab\pth{\wc{\Ga}+\breve{\Ga}} + (\wc{g}+\breve{g})\nab\bar{\Ga} + \bar{\Ga}\pth{\wc{\Ga}+\breve{\Ga}} + \bar{\Ga}^2(\wc{g}+\breve{g}) }\wc{\mathbb{U}}. 
\end{align*}
Note that we have omitted undifferentiated components of $\bar{g}$ since they are $O(1)$. We now explain the origin of the three remainders above: 
\begin{itemize}
\item $\RR_1\pth{\p}$ collects quadratic terms in $(\wc{g},\wc{\pi})+\p$ coming from the linearization of $\Phi\pth{\hat{g}(\p),\hat{\pi}(\p)}$ around $(\bar{g},\bar{\pi})$, i.e. $Q_\HH[\bar{g},\bar{\pi}](\wc{g} + \breve{g}, \wc{\pi} + \breve{\pi})$ and $Q_\MM[\bar{g},\bar{\pi}](\wc{g} + \breve{g}, \wc{\pi} + \breve{\pi})$ for which we use \eqref{QH}-\eqref{QM}.
\item $\RR_2\pth{\p,\wc{\mathbb{U}}}$ collects quadratic terms in $\wc{\mathbb{U}}$ produced after having extracted from $\Phi(g,\pi)=0$ the linear system for $\mathbb{U}$ (recall the definition \eqref{ansatz}), for instance $4(1+\wc{u})^{-1} \hat{g}(\p)^{j\ell} \pth{2 L_{\hat{g}(\p)}\wc{X}_{i\ell} \dr_j\wc{u} - L_{\hat{g}(\p)}\wc{X}_{j\ell} \dr_i \wc{u}}$ which is of the form $\hat{g}(\p) \pth{\nab \wc{\mathbb{U}}}^2$.
\item $\RR_3\pth{\p,\wc{\mathbb{U}}}$ collects quadratic cross terms in $\wc{\mathbb{U}}$ and $(\wc{g},\wc{\pi})+\p$ obtained when relating the linear operators defined with $(\hat{g}(\p),\hat{\pi}(\p))$ applied to $\wc{\mathbb{U}}$ with the ones defined with $(\bar{g},\bar{\pi})$, for instance $8\pth{\De_{\hat{g}(\p)}-\De_{\bar{g}}}\wc{u} $ which is of the form $(\wc{g}+\breve{g})\nab^2\wc{\mathbb{U}}+(\wc{\Ga}+\breve{\Ga})\nab \wc{\mathbb{U}}$.
\end{itemize}

In the lemmas below and in the rest of the article, all the implicit constants hidden with the symbol $\lesssim$ depend on the size of the support of $\breve{g}$ and $\breve{\pi}$, which is not an issue since this support will be fixed in Definition \ref{def space of correctors} below.

\begin{lemma}
  \label{lem reste 1}
Let $\de\in(0,1)$. Assume
\begin{align}
    \l \wc{\mathbb{U}}\r_{H^2_{-q-\de}(\Si)} & \leq C_0 \e, \label{assumption lem reste 1}
    \\ \l \breve{g}\r_{W^{2,\infty}(\Si)} + \l \breve{\pi}\r_{W^{1,\infty}(\Si)} & \leq D_0 \e,\label{assumption lem reste 2}
\end{align}
for some $C_0,D_0>0$. We have
\begin{align}
    \l \RR\pth{\p,\wc{\mathbb{U}}} \r_{L^2_{-q-\de-2}(\Si)} & \lesssim C(C_0,D_0)\e^2  \label{estim reste 1},
    \\ \left| \ps{\RR\pth{\p,\wc{\mathbb{U}}}, \mathbb{W}_{j,\ell,\a}}_{L^2(\Sigma)} \right| & \lesssim C(C_0,D_0)\e^2 \label{estim reste 2}.
\end{align}
\end{lemma}

\begin{proof}
  Recall that $\p=(\breve{g},\breve{\pi})$ is compactly supported and
  that its support will be fixed in the whole proof so that when
  estimating terms containing either $\breve{g}$ or $\breve{\pi}$ we can
  drop the weights at infinity. Moreover, it suffices to estimate
  $\RR_1\pth{\p}$, $\RR_2\pth{\p,\wc{\mathbb{U}}}$ and
  $\RR_3\pth{\p,\wc{\mathbb{U}}}$ since the higher order terms in
  $\RR\pth{\p,\wc{\mathbb{U}}}$ enjoy better decay and are smaller. We
  start with $\RR_1\pth{\p}$:
\begin{align*}
    \l \RR_1\pth{\p} \r_{L^2_{-q-\de-2}(\Si)} & \lesssim \l (\wc{g}+\breve{g})\nab\pth{\wc{\Ga}+\breve{\Ga}} \r_{L^2_{-q-\de-2}(\Si)} + \l (\wc{g}+\breve{g})\pth{\wc{\Ga}+\breve{\Ga}} \r_{L^2_{-q-\de}(\Si)} 
    \\&\quad + \l \pth{\wc{\Ga}+\breve{\Ga}}^2 \r_{L^2_{-q-\de-2}(\Si)} + \l(\wc{g}+\breve{g})^2 \r_{L^2_{-q-\de+2}(\Si)},
\end{align*}
where we used asymptotic flatness of $(\bar{g},\bar{\pi})$ from
Definition \ref{def AF}, which implies
$|\bar{\Ga}|\lesssim (1+r)^{-2}$. Using now the continuous bilinear
embeddings from Lemma \ref{lem embedding} we obtain
\begin{align*}
    \l \RR_1\pth{\p} \r_{L^2_{-q-\de-2}(\Si)} & \lesssim \l \wc{g}+\breve{g}\r_{H^2_{-q-\de}(\Si)}\pth{ \l\nab\pth{\wc{\Ga}+\breve{\Ga}} \r_{L^2_{-q-\de-2}(\Si)} + \l \wc{\Ga}+\breve{\Ga} \r_{H^1_{-q-\de-1}(\Si)} }
    \\&\quad + \l \wc{\Ga}+\breve{\Ga}\r_{H^1_{-q-\de-1}(\Si)}^2 + \l\wc{g}+\breve{g} \r_{H^2_{-q-\de}(\Si)}^2.
\end{align*}
Using now \eqref{assumption check} and \eqref{assumption lem reste 2}
this gives
\begin{align}
    \l \RR_1\pth{\p} \r_{L^2_{-q-\de-2}(\Si)} & \lesssim C(D_0)\e^2 . \label{estim R1}
\end{align}
We now estimate $\RR_2\pth{\p,\wc{\mathbb{U}}}$:
\begin{align*}
    \l \RR_2\pth{\p,\wc{\mathbb{U}}} \r_{L^2_{-q-\de-2}(\Si)} & \lesssim (1+C(D_0)\e)\pth{ \l  \pth{\nab \wc{\mathbb{U}}}^2 \r_{L^2_{-q-\de-2}(\Si)} + \l  \wc{\mathbb{U}}  \nab \wc{\mathbb{U}}  \r_{L^2_{-q-\de}(\Si)}  + \l  \wc{\mathbb{U}}^2 \r_{L^2_{-q-\de+2}(\Si)} },
\end{align*}
where we used \eqref{def g hat pi hat}, asymptotic flatness,
\eqref{assumption check} and \eqref{assumption lem reste 2} which
imply $\hat{g}(\p)=1+\GO{r^{-1}}+\GO{\e C(D_0) r^{-q-\de}}$ and
$\hat{\Ga}(\p)=\GO{r^{-2}} + \e C(D_0) \tilde{\Ga}(\p)$ where $\tilde{\Ga}(\p)\in H^1_{-q-\de-1}(\Si)$ (for this last term we thus need to use a multiplication property from Lemma \ref{lem embedding}). The
assumption \eqref{assumption lem reste 1} and the continuous bilinear
embeddings from Lemma \ref{lem embedding} then imply
\begin{align}
    \l \RR_2\pth{\p,\wc{\mathbb{U}}} \r_{L^2_{-q-\de-2}(\Si)} & \lesssim C(C_0,D_0)\e^2 . \label{estim R2}
\end{align}
We finally estimate $\RR_3\pth{\p,\wc{\mathbb{U}}}$:
\begin{align*}
  \l \RR_3\pth{\p,\wc{\mathbb{U}}} \r_{L^2_{-q-\de-2}(\Si)}
  & \lesssim \l (\wc{g}+\breve{g})\nab^2\wc{\mathbb{U}} \r_{L^2_{-q-\de-2}(\Si)} + \l \pth{\wc{\Ga}+\breve{\Ga}}\nab \wc{\mathbb{U}}\r_{L^2_{-q-\de-2}(\Si)} +  \l \pth{\wc{g}+\breve{g}}\nab \wc{\mathbb{U}}\r_{L^2_{-q-\de}(\Si)}
  \\
  &\quad + \l \nab\pth{\wc{\Ga}+\breve{\Ga}} \wc{\mathbb{U}}  \r_{L^2_{-q-\de-2}(\Si)} + \l (\wc{g}+\breve{g}) \wc{\mathbb{U}} \r_{L^2_{-q-\de+1}(\Si)}
  \\
  &\quad + \l \pth{\wc{\Ga}+\breve{\Ga}} \wc{\mathbb{U}} \r_{L^2_{-q-\de}(\Si)} + \l (\wc{g}+\breve{g}) \wc{\mathbb{U}} \r_{L^2_{-q-\de+2}(\Si)},
\end{align*}
where we used asymptotic flatness to write $\bar{\Ga}=\GO{r^{-2}}$ and
$\nab\bar{\Ga}=\GO{r^{-3}}$. As above the continuous bilinear
embeddings from Lemma \ref{lem embedding} together with \eqref{assumption check}, \eqref{assumption lem reste 1} and \eqref{assumption lem reste 2} imply
\begin{align}
    \l \RR_3\pth{\p,\wc{\mathbb{U}}} \r_{L^2_{-q-\de-2}(\Si)} & \lesssim C(C_0,D_0)\e^2 . \label{estim R3}
\end{align}
Putting \eqref{estim R1}, \eqref{estim R2} and \eqref{estim R3}
together concludes the proof of \eqref{estim reste 1}. Finally, since
$j\leq q$ and $\de>0$ imply
$\left|\ps{\mathbb{V},\mathbb{W}_{j,\ell,\a}}_{L^2(\Sigma)}\right|\lesssim \l
\mathbb{V}\r_{L^2_{-q-\de-2}}$ (where we also used
$|\mathbb{W}_{j,\ell,\a}|\lesssim r^{j-1}$), \eqref{estim reste 1}
implies \eqref{estim reste 2}.
\end{proof}

\begin{lemma}\label{lem reste 2}
  Let $\wc{\mathbb{U}},\wc{\mathbb{U}}'$ satisfying \eqref{assumption
    lem reste 1} and $\p,\p'$ satisfying \eqref{assumption lem reste
    2}. We have
  \begin{align}
    &\l \RR\pth{\p,\wc{\mathbb{U}}} - \RR\pth{\p',\wc{\mathbb{U}}'} \r_{L^2_{-q-\de-2}(\Si)} \label{estim reste 3}
    \\&\hspace{2cm} \lesssim  C(C_0,D_0)\e \bigg( \l \breve{g}-\breve{g}' \r_{W^{2,\infty}(\Si)} + \l \breve{\pi}-\breve{\pi}' \r_{W^{1,\infty}(\Si)}+\l \wc{\mathbb{U}} - \wc{\mathbb{U}}' \r_{H^2_{-q-\de}(\Si)} \bigg),\nonumber
    \\ &\left| \ps{\RR\pth{\p,\wc{\mathbb{U}}} - \RR\pth{\p',\wc{\mathbb{U}}'}, \mathbb{W}_{j,\ell,\a}}_{L^2(\Sigma)} \right| \label{estim reste 4}
    \\&\hspace{2cm} \lesssim C(C_0,D_0)\e \bigg( \l \breve{g}-\breve{g}' \r_{W^{2,\infty}(\Si)} + \l \breve{\pi}-\breve{\pi}' \r_{W^{1,\infty}(\Si)}+\l \wc{\mathbb{U}} - \wc{\mathbb{U}}' \r_{H^2_{-q-\de}(\Si)} \bigg).\nonumber
  \end{align}
\end{lemma}

\begin{proof}
  We have
  \begin{align*}
    \RR\pth{\p,\wc{\mathbb{U}}} - \RR\pth{\p',\wc{\mathbb{U}}'} & = \RR_1\pth{\p} - \RR_1\pth{\p'} + \RR_2\pth{\p,\wc{\mathbb{U}}} - \RR_2\pth{\p',\wc{\mathbb{U}}'} 
    \\&\quad + \RR_3\pth{\p,\wc{\mathbb{U}}} - \RR_3\pth{\p',\wc{\mathbb{U}}'} + \text{difference of higher order terms}.
  \end{align*}
  We have
  \begin{align*}
    \RR_1\pth{\p} - \RR_1\pth{\p'} & = (\breve{g}-\breve{g}')\nab\pth{\wc{\Ga}+\breve{\Ga}} + (\wc{g}+\breve{g}')\nab\pth{\breve{\Ga}-\breve{\Ga}'}
    \\&\quad + \bar{\Ga}\pth{(\breve{g}-\breve{g}')\pth{\wc{\Ga}+\breve{\Ga}} + (\wc{g}+\breve{g}')\pth{\breve{\Ga}-\breve{\Ga}'}}
    \\&\quad +  \pth{\breve{\Ga}-\breve{\Ga}'}\pth{\wc{\Ga}+\breve{\Ga}+\breve{\Ga}'} + \bar{\Ga}^2 \pth{\breve{g}-\breve{g}'}\pth{\wc{g}+\breve{g}+\breve{g}'},
  \end{align*}
  where this expression (as well as \eqref{diff R2 bis} and \eqref{diff R3 bis} below) is to be understood with constant coefficients in front of all the terms.
  By following the same lines as the proof of \eqref{estim R1} we obtain
  \begin{align}
    \l \RR_1\pth{\p} - \RR_1\pth{\p'} \r_{L^2_{-q-\de-2}(\Si)} & \lesssim C(C_0)\e \pth{ \l \breve{g}-\breve{g}' \r_{W^{2,\infty}(\Si)} + \l \breve{\pi}-\breve{\pi}' \r_{W^{1,\infty}(\Si)} }. \label{diff R1}
  \end{align}
  Next, we have
  \begin{equation}\label{diff R2 bis}
  \begin{aligned}
    \RR_2\pth{\p,\wc{\mathbb{U}}} - \RR_2\pth{\p',\wc{\mathbb{U}}'} & = (\breve{g}-\breve{g}')\pth{\nab \wc{\mathbb{U}}}^2 + \hat{g}(\p')\nab\pth{\wc{\mathbb{U}}-\wc{\mathbb{U}}'} \nab\pth{\wc{\mathbb{U}}+\wc{\mathbb{U}}'}
    \\&\quad + \pth{\breve{\Ga}-\breve{\Ga}'} \wc{\mathbb{U}}  \nab \wc{\mathbb{U}} + \hat{\Ga}(\p') \pth{ \pth{\wc{\mathbb{U}}-\wc{\mathbb{U}}'}  \nab \wc{\mathbb{U}} +\wc{\mathbb{U}}'  \nab \pth{\wc{\mathbb{U}}- \wc{\mathbb{U}}'}}
    \\&\quad + \pth{ \pth{\breve{\Ga}-\breve{\Ga}'}  \wc{\mathbb{U}} + \hat{\Ga}(\p') \pth{ \wc{\mathbb{U}} -\wc{\mathbb{U}}'} } \pth{\hat{\Ga}(\p)  \wc{\mathbb{U}} + \hat{\Ga}(\p')  \wc{\mathbb{U}}'}.
  \end{aligned}
  \end{equation}
  By following the same lines as the proof of \eqref{estim R2} we obtain
  \begin{align}
    \l \RR_2\pth{\p,\wc{\mathbb{U}}} - \RR_2\pth{\p',\wc{\mathbb{U}}'}\r_{L^2_{-q-\de-2}(\Si)}
    \lesssim{}&  C(C_0,D_0)\e \l \wc{\mathbb{U}} - \wc{\mathbb{U}}' \r_{H^2_{-q-\de}(\Si)}\label{diff R2}
    \\
              & + C(C_0,D_0)\e^2 \pth{ \l \breve{g}-\breve{g}' \r_{W^{2,\infty}(\Si)} + \l \breve{\pi}-\breve{\pi}' \r_{W^{1,\infty}(\Si)} }\nonumber.
  \end{align}
  Finally, we have
  \begin{equation}\label{diff R3 bis}
  \begin{aligned}
    \RR_3\pth{\p,\wc{\mathbb{U}}} - \RR_3\pth{\p',\wc{\mathbb{U}}'}
    ={}& (\breve{g}-\breve{g}')\nab^2\wc{\mathbb{U}} + (\wc{g}+\breve{g}')\nab^2\pth{\wc{\mathbb{U}}-\wc{\mathbb{U}}' }
    \\
    & + \pth{\breve{\Ga}-\breve{\Ga} + (\breve{g} - \breve{g}')\bar{\Ga}} \nab \wc{\mathbb{U}}  + \pth{\wc{\Ga}+\breve{\Ga}' + (\wc{g}+\breve{g}')\bar{\Ga}}\nab \pth{\wc{\mathbb{U}}-\wc{\mathbb{U}}'}  
    \\
    & + \pth{ \nab\pth{\breve{\Ga}-\breve{\Ga}'} + (\breve{g}-\breve{g}')\nab\bar{\Ga} + \bar{\Ga}\pth{\breve{\Ga}-\breve{\Ga}'} + \bar{\Ga}^2(\breve{g}-\breve{g}') }\wc{\mathbb{U}} 
    \\
    & + \pth{ \nab\pth{\wc{\Ga}+\breve{\Ga}'} + (\wc{g}+\breve{g}')\nab\bar{\Ga} + \bar{\Ga}\pth{\wc{\Ga}+\breve{\Ga}'} + \bar{\Ga}^2(\wc{g}+\breve{g}') }\pth{\wc{\mathbb{U}}-\wc{\mathbb{U}}'}.
  \end{aligned}
  \end{equation}
  By following the same lines as the proof of \eqref{estim R3} we obtain
  \begin{equation}
    \begin{split}
      &\l \RR_3\pth{\p,\wc{\mathbb{U}}} - \RR_3\pth{\p',\wc{\mathbb{U}}'}\r_{L^2_{-q-\de-2}(\Si)}  \label{diff R3}
    \\
    \lesssim{}& C(C_0,D_0)\e \bigg( \l \breve{g}-\breve{g}' \r_{W^{2,\infty}(\Si)} + \l \breve{\pi}-\breve{\pi}' \r_{W^{1,\infty}(\Si)}+\l \wc{\mathbb{U}} - \wc{\mathbb{U}}' \r_{H^2_{-q-\de}(\Si)} \bigg). 
    \end{split}    
  \end{equation}
  By putting \eqref{diff R1}, \eqref{diff R2} and \eqref{diff R3} together we obtain \eqref{estim reste 3}. We conclude the proof of the lemma by noting that \eqref{estim reste 3} implies \eqref{estim reste 4}, as \eqref{estim reste 1} implies \eqref{estim reste 2}.
\end{proof}

\subsection{Solving for the corrector}\label{section corrector} 

In this section, we construct the corrector $\p$, starting by deducing from the analyticity assumptions on $\pth{\Om,\bar{g}_{|_\Om},\bar{\pi}_{|_\Om}}$ (see Theorem \ref{theo general}) that KIDS, i.e. elements of $\ker(D\Phi[\bar{g},\bar{\pi}]^*)$, enjoy similar regularity.

\begin{lemma}\label{lem regularity kids}
If $D\Phi[\bar{g},\bar{\pi}]^*(f,X)=0$ on a neighborhood of $\dr\Si$, then $f$, $X$, $\dr_\nu f$ and $D_\nu X$ are analytic on $\dr\Si$.
\end{lemma}

\begin{proof}
If $D\Phi[\bar{g},\bar{\pi}]^*(f,X)=0$ on a neighborhood of $\dr\Si$, then (recall \eqref{KIDS 1}-\eqref{KIDS 2}) $(f,X)$ satisfies a system of the form
\begin{align*}
    \Hess_{\bar{g}}(f) & = B^{(1)}_{k\ell} D^k X^\ell + B^{(2)} f + B^{(3)}_k X^k,
    \\  \LL_X\bar{g} & = A^{(2)} f,
\end{align*}
where the coefficients $B^{(1)}$, $B^{(2)}$, $B^{(3)}$ and $A^{(2)}$ depend only on $\bar{g}$ and $\bar{\pi}$ (and their derivatives) and are analytic on $\{r_0\leq r < r_1\}$. To simplify the exposition, in what follows we do not write the analytic coefficients of the equations satisfied by $f$ and $X$. In particular the above equations rewrite schematically as
\begin{align}
    \Hess_{\bar{g}}(f)  & = D X +  f + X,\label{eq kids f}
    \\  \LL_X\bar{g} & = f.\label{eq kids X}
\end{align}
We now introduce the standard
notations from Chapter 3 of \cite{Christodoulou1993}, i.e.
$\nabb$, $\Hessb$ and $\lapp$ denote covariant derivative, Hessian and
Laplace-Beltrami operator associated to the induced metric
$\gb$ on $\dr\Si$ by $\bar{g}$ and for any tangential vector field $\Yb$ to $\dr\Si$ we define $\nabb_\nu\Yb$ to be the projection on $T\dr\Si$ of $D_\nu\Yb$, i.e.
\begin{align*}
\nabb_\nu\Yb \vcentcolon = D_\nu\Yb - \bar{g}\pth{ D_\nu\Yb,\nu}\nu.
\end{align*}
We also decompose the vector field $X=N\nu + \Xb$ where $\Xb\in T\dr\Si$ and $N$ is a scalar function. In order to project the equations \eqref{eq kids f}-\eqref{eq kids X} on $\dr\Si$ we will use the following exact formulas:
\begin{align}
\Hess_{\bar{g}}(f)(\nu,\nu) & = \dr_\nu^2 f - \gb^{ab}D_\nu\nu_a e_b(f),\label{AA}
\\ \Hess_{\bar{g}}(f)(e_a,e_b) & = \Hessb_{\gb}(f)(e_a,e_b) +  \bar{g}(D_{e_a}\nu,e_b)\dr_\nu f,
\\ \pth{ \LL_X\bar{g}}(\nu,\nu) & = 2 \pth{ \dr_\nu N - \gb(\Xb,D_\nu\nu) },
\\ \pth{ \LL_X\bar{g}}(\nu,e_a) & =  N \gb(D_\nu\nu,e_a) + \gb(\nabb_\nu \Xb,e_a)  + e_a(N) -\gb(\Xb,D_{e_a}\nu),
\\ \pth{ \LL_X\bar{g}}(e_a,e_b) & = \gb(\nabb_{e_a}\Xb,e_b) + \gb(\nabb_{e_b}\Xb,e_a) + N\pth{ \LL_\nu\bar{g}}(e_a,e_b), \label{ZZ}
\end{align}
where $(e_1,e_2)$ is basis of $T\dr\Si$ and the indexes $a,b$ run through $1,2$. As we did when we obtained \eqref{eq kids f}-\eqref{eq kids X}, we do not write the exact coefficients which are all analytic by assumption. With this in mind, we use \eqref{eq kids f}-\eqref{eq kids X} to schematically rewrite the relations \eqref{AA}-\eqref{ZZ} as
\begin{align}
\dr_\nu^2 f & = \nabb f + \dr_\nu N + \nabb N + \nabb_\nu \Xb  + \nabb \Xb+  f + N + \Xb,\label{dnu f carré}
\\  \Hessb_{\gb}(f)(e_a,e_b) & =  \dr_\nu f + \dr_\nu N + \nabb N + \nabb_\nu \Xb  + \nabb \Xb +  f + N + \Xb, \label{hess tangent}
\\ \dr_\nu N & =  f + \Xb, \label{d nu N}
\\  \nabb_a  N & = \nabb_\nu \Xb_a + f + \Xb + N, \label{d nu Xb}
\\ \nabb_a \Xb_b+\nabb_b \Xb_a  & = f + N, \label{Lie tangent}
\end{align}
We take the tangential divergence of \eqref{d nu Xb} to get
\begin{align}\label{lap N}
    \lapp N & = \nabb^{\leq 1}( \nabb_\nu \Xb + f +  N + \Xb).
\end{align}
We take the tangential trace of \eqref{Lie tangent} to obtain $\divb \Xb= f+N$ and then take the tangential divergence of \eqref{Lie tangent} to obtain the following elliptic equation on $\Xb$
\begin{align}\label{lap Xb}
    \lapp \Xb & = \nabb^{\leq 1} (f + N + \Xb).
\end{align}
Using Corollary 3.2.3.2 from \cite{Christodoulou1993}, $[\lapp,\dr_\nu]$ involves $\nabb \dr_\nu$, $\nabb^2$, $\dr_\nu$ and $\nabb$. Therefore we obtain from \eqref{lap Xb} the following
\begin{align}\label{lap d nu Xb}
    \lapp \nabb_\nu \Xb & = \nabb^{\leq 1}(f + N + \Xb + \nabb \Xb + \dr_\nu f +  \nabb_\nu \Xb),
\end{align}
where we have used \eqref{d nu N} to replace $\dr_\nu N$. We also commute \eqref{lap Xb} with $\nabb$ to obtain
\begin{align*}
    \lapp \nabb \Xb & = \nabb^{\leq 1} (f + N + \Xb) + \nabb^2(f + N + \Xb).
\end{align*}
Using \eqref{d nu Xb} we replace $\nabb^2N$ by $\nabb \nabb_\nu \Xb$ and $\nabb N$ to obtain
\begin{align}\label{lap nabb Xb inter}
    \lapp \nabb \Xb & = \nabb^{\leq 1} (f + N + \Xb + \nabb_\nu \Xb + \nabb \Xb) + \nabb^2 f.
\end{align}
We use \eqref{d nu N} and \eqref{d nu Xb} to replace $\dr_\nu N$ and $\nabb_\nu\Xb$ in \eqref{hess tangent} and obtain
\begin{align}\label{hessb f}
    \Hessb f & = \dr_\nu f + \nabb^{\leq 1}( f + N + \Xb).
\end{align}
Taking the tangential trace of \eqref{hessb f} we obtain the equation for $f$:
\begin{align}\label{lap f}
    \lapp f & = \dr_\nu f + \nabb^{\leq 1}( f + N + \Xb).
\end{align}
Moreover, by using \eqref{hessb f} we can also control $\nabb^2 f$ in \eqref{lap nabb Xb inter} which becomes
\begin{align}\label{lap nabb Xb}
    \lapp \nabb \Xb & = \dr_\nu f + \nabb^{\leq 1} (f + N + \Xb + \nabb_\nu \Xb + \nabb \Xb).
\end{align}
Also, we commute \eqref{lap f} with $\dr_\nu$ (using again Corollary 3.2.3.2 from \cite{Christodoulou1993}) to obtain an equation on $\dr_\nu f$. We get:
\begin{align}\label{lap d nu f inter}
    \lapp \dr_\nu f & =   \dr_\nu^2 f +  \nabb^{\leq 1}( f + N + \Xb + \dr_\nu f  + \nabb_\nu \Xb),
\end{align}
where we have used \eqref{d nu N} and \eqref{hessb f} to replace $\dr_\nu N$ and $\nabb^2 f$. Using \eqref{dnu f carré} we can replace $\dr_\nu^2$ in \eqref{lap d nu f inter} to finally obtain
\begin{align}\label{lap d nu f}
    \lapp \dr_\nu f & =   \nabb^{\leq 1}( f + N + \Xb + \dr_\nu f  + \nabb_\nu \Xb).
\end{align}
By using the schematic notation
\begin{align*}
    V =\{ f , N , \Xb, \dr_\nu f , \nabb_\nu \Xb, \nabb\Xb \},
\end{align*}
we remark that combining equations \eqref{lap N}, \eqref{lap Xb}, \eqref{lap d nu Xb}, \eqref{lap f}, \eqref{lap nabb Xb} and \eqref{lap d nu f} gives an elliptic system on $\dr\Si$ of the form
\begin{align*}
    \lapp V = \nabb^{\leq 1} V.
\end{align*}
In particular, the coefficients of this system are analytic by assumption and thus standard elliptic regularity (see Theorem 7.5.1 in \cite{Hormanderzero}, and note that $\dr\Si$ is a manifold without boundary) implies that $V$ is analytic on $\dr\Si$. This concludes the proof of Lemma \ref{lem regularity kids}.
\end{proof}

We now define the space $\mathcal{Z}_q$ of correctors.

\begin{definition}\label{def space of correctors}
We fix $I\subset (r_0,r_1)$ an open interval in $\RRR$.
\begin{enumerate}[label=(\roman*)]
    \item We define $\chi_{corr}:\RRR\longrightarrow[0,1]$ to be a smooth function such that $\chi_{corr|_{\RRR\setminus I}}=0$ and $\chi_{corr|_{I}}>0$.
    \item For $j\geq 1$, $-(j-1)\leq \ell \leq j-1$ and $\a=0,1,2,3$ we define 
    \begin{align*}
        Z_{j,\ell,\a} & \vcentcolon =\chi_{corr}(r) D\Phi[\bar{g},\bar{\pi}]^*\pth{\mathbb{W}_{j,\ell,\a}}.
    \end{align*}
    \item For $q\in \mathbb{N}^*$ we define the $4q^2$ dimensional space of correctors
    \begin{align*}
        \mathcal{Z}_q & \vcentcolon = \mathrm{span}\big\{ Z_{j,\ell,\a}, \text{$1\leq j\leq q$, $-(j-1)\leq \ell \leq j-1$ and $\a=0,1,2,3$}\big\},
    \end{align*}
    endowed with the norm $\l\cdot\r_{W^{2,\infty}(\Si)\times W^{1,\infty}(\Si)}$.
\end{enumerate}
\end{definition}

The following lemma will allow us to solve the equation for the corrector.

\begin{lemma}\label{lem KIDS}
  Let $q\in\mathbb{N}^*$. 
  \begin{enumerate}[label=(\roman*)]
  \item The matrix
    \begin{align*}
      \pth{\ps{Z_{j,\ell,\a},D\Phi[\bar{g},\bar{\pi}]^*\pth{\mathbb{W}_{j',\ell',\a'}}}_{L^2(\Sigma)}}_{1\leq j,j'\leq q,-(j-1)\leq \ell\leq j-1,-(j'-1)\leq\ell'\leq j'-1, 0\leq \a,\a'\leq 3}
    \end{align*}
    is invertible. 
  \item Let $\mu_{j,\ell,\a}$ be real numbers for $1\leq j \leq q$,
    $-(j-1)\leq \ell \leq j-1$ and $\a=0,1,2,3$. There exists a unique
    $\p \in \mathcal{Z}_q$ such that
    \begin{align*}
      \ps{\p,D\Phi[\bar{g},\bar{\pi}]^*\pth{\mathbb{W}_{j,\ell,\a}}}_{L^2(\Sigma)} = \mu_{j,\ell,\a},
    \end{align*}
    and moreover $\p$ satisfies
    \begin{align*}
      \l \p \r_{W^{2,\infty}(\Si)\times W^{1,\infty}(\Si) } \lesssim \sup_{j,\ell,\a}|\mu_{j,\ell,\a}|.
    \end{align*}
  \end{enumerate}
\end{lemma}

\begin{proof}
  We start with the first point of the lemma. Let
  $\la^{j,\ell,\a}$ be real numbers defined such that
  \begin{align*}
    \ps{\la^{j,\ell,\a}\chi_{corr}(r) D\Phi[\bar{g},\bar{\pi}]^*\pth{\mathbb{W}_{j,\ell,\a}},D\Phi[\bar{g},\bar{\pi}]^*\pth{\mathbb{W}_{j',\ell',\a'}}}_{L^2(\Sigma)} = 0, 
  \end{align*}
  for all $1\leq j'\leq q$, $-(j'-1)\leq \ell'\leq j'-1$ and
  $\a'=0,1,2,3$. We multiply each equality by $\la^{j',\ell',\a'}$ and
  sum over the indices $j'$, $\ell'$ and $\a'$ to get
  \begin{align*}
    \ps{\chi_{corr}  D\Phi[\bar{g},\bar{\pi}]^*\pth{\la^{j,\ell,\a}\mathbb{W}_{j,\ell,\a}}, D\Phi[\bar{g},\bar{\pi}]^*\pth{\la^{j,\ell,\a}\mathbb{W}_{j,\ell,\a}}}_{L^2(\Sigma)} = 0, 
  \end{align*}
  Since $\chi_{corr|_{\RRR\setminus I}}= 0$ and $\chi_{corr|_{ I}}>0$ this implies
  \begin{align}\label{identity W}
    D\Phi[\bar{g},\bar{\pi}]^*\pth{\la^{j,\ell,\a} \mathbb{W}_{j,\ell,\a}} = 0, \quad \text{on $\{r\in I\}$.}
  \end{align}
  Recall from Lemma \ref{lem harmonic polynomials} that
  $P^*\pth{\mathbb{W}_{j,\ell,\a}}=0$ on $\Si$ and by assumption (see
  Theorem \ref{theo general}) the coefficients of $P^*$ are analytic
  in the region $\{r_0<r<r_1\}$. Standard interior elliptic regularity (see Theorem 7.5.1 in \cite{Hormanderzero}) thus
  implies that $\mathbb{W}_{j,\ell,\a}$ is analytic in the region
  $\{r_0<r<r_1\}$. Since the assumptions also imply that the
  coefficients of $D\Phi[\bar{g},\bar{\pi}]^*$ are analytic in the
  region $\{r_0<r<r_1\}$, we can extend \eqref{identity W} to this
  region, and by continuity to $\{r_0\leq r<r_1\}$. We have thus proved
  \begin{align}\label{identity W 2}
    D\Phi[\bar{g},\bar{\pi}]^*\pth{\la^{j,\ell,\a} \mathbb{W}_{j,\ell,\a}} = 0, \quad \text{on $\{r_0\leq r<r_1\}$.}
  \end{align}
  Now, Lemma \ref{lem regularity kids} implies that $\la^{j,\ell,\a} \mathbb{W}_{j,\ell,\a}$ and $B_\nu^*(\la^{j,\ell,\a} \mathbb{W}_{j,\ell,\a})$ are analytic on the boundary
  $\dr\Si$. Recall from Proposition \ref{prop P final} that
  $(B_\nu^*+F_\de^*)(\mathbb{W}_{j,\ell,\a})=0$ on $\dr\Si$. Restricting
  this to the open subset $\mathcal{U}$ of $\dr\Si$ where $F_\de=0$ (which implies that $F_\de^*=0$ on the same open subset) we
  obtain
  $B_\nu^*(\la^{j,\ell,\a}
  \mathbb{W}_{j,\ell,\a})_{|_{\mathcal{U}}}=0$. Since
  $B_\nu^*(\la^{j,\ell,\a} \mathbb{W}_{j,\ell,\a})$ is analytic on the
  boundary, we extend this equality to
  $B_\nu^*(\la^{j,\ell,\a} \mathbb{W}_{j,\ell,\a})=0$ on
  $\dr\Si$. Returning to the boundary conditions, this implies that
  $F_\de^*( \la^{j,\ell,\a} \mathbb{W}_{j,\ell,\a})=0 $ on $\dr\Si$. Restricting this to the open subset $\mathcal{V}$ of $\dr\Si$ where $F_{\de}=\mathrm{Id}$ (which implies that $F_\de^*=\mathrm{Id}$ on the same open subset), we obtain $\la^{j,\ell,\a} \mathbb{W}_{j,\ell,\a |_{\mathcal{V}}}=0$
  which again extends to $\la^{j,\ell,\a} \mathbb{W}_{j,\ell,\a}=0$ on
  $\dr\Si$ thanks to analyticity. Together with Lemma \ref{lem
    harmonic polynomials} this implies that
  $P^*( \la^{j,\ell,\a} \mathbb{W}_{j,\ell,\a})=0$ on $\Si$,
  $B_\nu^*( \la^{j,\ell,\a} \mathbb{W}_{j,\ell,\a})= \la^{j,\ell,\a}
  \mathbb{W}_{j,\ell,\a}=0$ on $\dr\Si$. Lemma \ref{lem unique
    continuation} thus implies that
  \begin{align}\label{W=0}
    \la^{j,\ell,\a} \mathbb{W}_{j,\ell,\a} = 0, \quad \text{on $\Si$.}
  \end{align}
  Since the family formed by the $\mathbb{W}_{j,\ell,\a}$ is linearly
  independent on $\Si$, this implies that $\la^{j,\ell,\a}=0$ for all indices
  and concludes the proof of \Cref{lem KIDS}, since the second point follows
  directly from the first one.
\end{proof}

\begin{remark}
  The proof of Lemma \ref{lem KIDS} shows how the KIDS are
  obstructions to the construction of the corrector. In order to go
  from \eqref{identity W 2} to \eqref{W=0}, we need to use a property
  distinguishing elements of $\ker(D\Phi[\bar{g},\bar{\pi}]^*)$ and
  elements of $\ker(P^*_{F_\de,q-1+\de})$. The key is to note that
  $P^*(\mathbb{U})=0$ for all
  $\mathbb{U}\in\ker(D\Phi[\bar{g},\bar{\pi}]^*)$, which follows from
  the following alternative definition of $P$
  \begin{align*}
    P = D\Phi[\bar{g},\bar{\pi}]\circ \Pi,
  \end{align*}
  with $\Pi(u,X)=\pth{ 4u\bar{g}, 2u\bar{\pi} + L_{\bar{g}}X}$. This
  implies that the only way to distinguish
  $\ker(D\Phi[\bar{g},\bar{\pi}]^*)$ from $\ker(P^*_{F_\de,q-1+\de})$ is
  the boundary condition, and in particular the regularity at the boundary.
\end{remark}

We now construct the corrector map. If $\wc{\mathbb{U}}$ is such that 
\begin{align}\label{assumption u X}
  \l\wc{\mathbb{U}}\r_{H^2_{-q-\de}(\Si)}  \leq C_0 \e,
\end{align}
we want the corrector $\p\in \mathcal{Z}_q$ to be such that the
following orthogonality conditions hold
\begin{align}\label{orthogonality condition}
  \ps{ D\Phi[\bar{g},\bar{\pi}]((\wc{g},\wc{\pi})+\p) + \RR\pth{\p,\wc{\mathbb{U}}} , \mathbb{W}_{j,\ell,\a} }_{L^2(\Sigma)} & = 0 ,
\end{align}
for all $1\leq j \leq q$, $-(j-1)\leq \ell \leq j-1$ and
$\a=0,1,2,3$. According to Proposition \ref{prop P final} this will
ensure the solvability of \eqref{eq u X}. Since each
$\p\in\mathcal{Z}_q$ is compactly supported in $\Si$ we can perform
the following integration by parts without picking up boundary terms
\begin{align*}
  \ps{D\Phi[\bar{g},\bar{\pi}](\p),\mathbb{W}_{j,\ell,\a}}_{L^2(\Sigma)} = \ps{\p,D\Phi[\bar{g},\bar{\pi}]^*(\mathbb{W}_{j,\ell,\a})}_{L^2(\Sigma)} ,
\end{align*}
so that \eqref{orthogonality condition} rewrites
\begin{align}
  \ps{\p,D\Phi[\bar{g},\bar{\pi}]^*(\mathbb{W}_{j,\ell,\a})}_{L^2(\Sigma)} & =-\ps{ D\Phi[\bar{g},\bar{\pi}](\wc{g},\wc{\pi}),\mathbb{W}_{j,\ell,\a}}_{L^2(\Sigma)} -\ps{\RR\pth{\p,\wc{\mathbb{U}}} , \mathbb{W}_{j,\ell,\a} }_{L^2(\Sigma)} .
\end{align}

\begin{definition}
  Let $D_0>0$. The \emph{corrector map} $\Psi^{(c)}$ is defined by 
  \begin{align*}
    \Psi^{(c)} : B_{\mathcal{Z}_q}(0,D_0\e) & \longrightarrow B_{\mathcal{Z}_q}(0,D_0\e),
    \\ \p & \longmapsto \Psi^{(c)}(\p),
  \end{align*}
  such that $\Psi^{(c)}(\p)$ solves
  \begin{align}\label{eq p}
    \ps{\Psi^{(c)}(\p),D\Phi[\bar{g},\bar{\pi}]^*(\mathbb{W}_{j,\ell,\a})}_{L^2(\Sigma)} & =-\ps{ D\Phi[\bar{g},\bar{\pi}](\wc{g},\wc{\pi}),\mathbb{W}_{j,\ell,\a}}_{L^2(\Sigma)} -\ps{\RR\pth{\p,\wc{\mathbb{U}}} , \mathbb{W}_{j,\ell,\a} }_{L^2(\Sigma)}
  \end{align}
  for all $1\leq j \leq q$, $-(j-1)\leq \ell \leq j-1$ and $\a=0,1,2,3$.
\end{definition}

\begin{lemma}\label{lem corrector map}
  Assume that $\wc{\mathbb{U}}$ satisfies \eqref{assumption u X}. If $D_0$ is
  large enough and if $\e$ is small enough compared to $D_0^{-1}$ and
  $C_0^{-1}$, then the corrector map $\Psi^{(c)}$ is a well-defined
  contraction.
\end{lemma}

\begin{proof}
  Let $\p\in B_{\mathcal{Z}_q}(0,D_0\e)$. Thanks to Lemma \ref{lem
    KIDS} there exists a unique $\Psi^{(c)}(\p)\in\mathcal{Z}_q$
  satisfying \eqref{eq p} for all $1\leq j \leq q$,
  $-(j-1)\leq \ell \leq j-1$ and $\a=0,1,2,3$ satisfying moreover
  \begin{align*}
    \l \Psi^{(c)}(\p) \r_{W^{2,\infty}(\Si)\times W^{1,\infty}(\Si) } 
    \lesssim \sup_{j,\ell,\a} \left|
    \ps{ D\Phi[\bar{g},\bar{\pi}](\wc{g},\wc{\pi}),\mathbb{W}_{j,\ell,\a}}_{L^2(\Sigma)} 
    +\ps{\RR\pth{\p,\wc{\mathbb{U}}} , \mathbb{W}_{j,\ell,\a} }_{L^2(\Sigma)} \right|.
  \end{align*}
  We estimate both terms on the right-hand side. First since $j\leq q$ and $\de>0$ we have
  \begin{align*}
    \left|\ps{ D\Phi[\bar{g},\bar{\pi}](\wc{g},\wc{\pi}),\mathbb{W}_{j,\ell,\a}}\right| & \lesssim \l D\Phi[\bar{g},\bar{\pi}](\wc{g},\wc{\pi}) \r_{L^2_{-q-\de-2}(\Si)}
    \\&\lesssim \l \wc{g} \r_{H^2_{-q-\de}(\Si)} + \l \wc{\pi} \r_{H^1_{-q-\de-1}(\Si)}
    \\&\lesssim \e
  \end{align*}
  where we have used \eqref{DH}, \eqref{DM}, asymptotic flatness and
  \eqref{assumption check}. For the second term in the estimate for
  $\Psi^{(c)}(\p)$ we use \eqref{estim reste 2}. We obtain
  \begin{align*}
    \l \Psi^{(c)}(\p) \r_{W^{2,\infty}(\Si)\times W^{1,\infty}(\Si) } \lesssim \e\pth{ 1 + C(C_0,D_0)\e} .
  \end{align*}
  Taking $D_0$ large enough and $\e$ small enough compared to
  $D_0^{-1}$ and $C_0^{-1}$ shows that
  $\Psi^{(c)}(\p)\in B_{\mathcal{Z}_q}(0,D_0\e)$ and thus that the
  corrector map $\Psi^{(c)}$ is well-defined. 
  
  Next let
  $\p,\p'\in B_{\mathcal{Z}_q}(0,D_0\e)$. Taking the difference of the
  relations defining $\Psi^{(c)}(\p)$ and $\Psi^{(c)}(\p')$ and using
  Lemma \ref{lem KIDS} we obtain
  \begin{align*}
    \l \Psi^{(c)}(\p) - \Psi^{(c)}(\p')  \r_{W^{2,\infty}(\Si)\times W^{1,\infty}(\Si) }
    & \lesssim \sup_{j,\ell,\a}\left| \ps{\RR\pth{\p,\wc{\mathbb{U}}} - \RR\pth{\p',\wc{\mathbb{U}}} , \mathbb{W}_{j,\ell,\a} }_{L^2(\Sigma)} \right|
    \\
    &\lesssim C(C_0,D_0)\e \l \p-\p'\r_{W^{2,\infty}(\Si)\times W^{1,\infty}(\Si)} ,
  \end{align*}
  where we have used \eqref{estim reste 4}. Taking $\e$ small enough compared to $D_0^{-1}$ and $C_0^{-1}$ shows that $\Psi^{(c)}$ is a contraction.
\end{proof}

From now on, $D_0>0$ is fixed according to Lemma \ref{lem corrector map}.

\begin{corollary}\label{coro corrector}
  Assume that $\wc{\mathbb{U}}$ satisfies \eqref{assumption u X}. If $\e$ is
  small enough compared to $C_0^{-1}$, there exists a unique
  $\p\pth{\wc{\mathbb{U}}}\in \mathcal{Z}_q$ such that
  \eqref{orthogonality condition} holds. Moreover we have
  \begin{align}\label{estim p(U)}
    \l \p\pth{\wc{\mathbb{U}}} \r_{W^{2,\infty}(\Si)\times W^{1,\infty}(\Si)} \lesssim \e 
  \end{align}
  and
  \begin{align}\label{estim p(U)-p(U')}
    \l\p\pth{\wc{\mathbb{U}}} - \p\pth{\wc{\mathbb{U}}'}  \r_{W^{2,\infty}(\Si)\times W^{1,\infty}(\Si)  } \lesssim C(C_0)\e \l \wc{\mathbb{U}}-\wc{\mathbb{U}}'\r_{H^2_{-q-\de}(\Si)},
  \end{align}
  if $\wc{\mathbb{U}}$ and $\wc{\mathbb{U}}'$ both satisfy \eqref{assumption u X}.
\end{corollary}

\begin{proof}
  The proof of the first part of the corollary follows from Lemma
  \ref{lem corrector map} and the Banach fixed point theorem. For the
  proof of \eqref{estim p(U)-p(U')} we take the difference of the
  relations defining $\p\vcentcolon=\p\pth{\wc{\mathbb{U}}}$ and
  $\p'\vcentcolon=\p\pth{\wc{\mathbb{U}}'}$ together with Lemma
  \ref{lem KIDS} and \eqref{estim reste 4} to obtain
  \begin{align*}
    \l\p - \p'  \r_{W^{2,\infty}(\Si)\times W^{1,\infty}(\Si) }
    & \lesssim \sup_{j,\ell,\a}\left| \ps{\RR\pth{\p,\wc{\mathbb{U}}} - \RR\pth{\p',\wc{\mathbb{U}}'} , \mathbb{W}_{j,\ell,\a} }_{L^2(\Sigma)} \right|
    \\
    &\lesssim C(C_0)\e \bigg( \l\p - \p'  \r_{W^{2,\infty}(\Si)\times W^{1,\infty}(\Si) } +\l \wc{\mathbb{U}} - \wc{\mathbb{U}}' \r_{H^2_{-q-\de}(\Si)} \bigg).
  \end{align*}
  By taking $\e$ small enough compared to $C_0^{-1}$ we can absorb
  $C(C_0)\e \l\p - \p' \r_{W^{2,\infty}(\Si)\times W^{1,\infty}(\Si)}
  $ by the LHS and obtain \eqref{estim p(U)-p(U')}.
\end{proof}

\subsection{Solving the elliptic system}
\label{section solving system}

In this section we conclude the proof of Theorem \ref{theo general} by
solving the elliptic system for $\wc{\mathbb{U}}$ with the corrector
$\p\pth{\wc{\mathbb{U}}}$ constructed in Section \ref{section
  corrector}. We want to solve
\begin{equation}\label{system u X}
  \left\{
    \begin{aligned}
      P\pth{\wc{\mathbb{U}}} & = D\Phi[\bar{g},\bar{\pi}]\pth{(\wc{g},\wc{\pi})+\p\pth{\wc{\mathbb{U}}}} + \RR\pth{\p\pth{\wc{\mathbb{U}}},\wc{\mathbb{U}}},  \quad \text{on $\Si$,}
      \\ (B_\nu + F_\de)\pth{\wc{\mathbb{U}}} & = 0 ,  \quad \text{on $\dr\Si$.}
    \end{aligned}
  \right.
\end{equation}

\begin{definition}
  Let $C_0>0$. The \emph{solution map} $\Psi^{(s)}$ is defined by
  \begin{align*}
    \Psi^{(s)} : B_{H^2_{-q-\de}(\Si)}(0,C_0\e) & \longrightarrow B_{H^2_{-q-\de}(\Si)}(0,C_0\e)
    \\ \wc{\mathbb{U}} & \longmapsto \Psi^{(s)} \pth{\wc{\mathbb{U}}}
  \end{align*}
  such that $\Psi^{(s)}\pth{\wc{\mathbb{U}}}$ solves
  \begin{equation}\label{system psi(u,X)}
    \left\{
      \begin{aligned}
        P\pth{\Psi^{(s)} \pth{\wc{\mathbb{U}}}} & = D\Phi[\bar{g},\bar{\pi}]\pth{(\wc{g},\wc{\pi})+\p\pth{\wc{\mathbb{U}}}} + \RR\pth{\p\pth{\wc{\mathbb{U}}},\wc{\mathbb{U}}},  \quad \text{on $\Si$,}
        \\ (B_\nu + F_\de)\pth{\Psi^{(s)} \pth{\wc{\mathbb{U}}}} & = 0 ,  \quad \text{on $\dr\Si$.}
      \end{aligned}
    \right.
  \end{equation}
\end{definition}

\begin{lemma}\label{dernier lemme}
  If $C_0$ is large enough and $\e$ is small enough compared to
  $C_0^{-1}$, the solution map $\Psi^{(s)}$ is a well-defined
  contraction.
\end{lemma}

\begin{proof}
  Let $\wc{\mathbb{U}}\in B_{H^2_{-q-\de}(\Si)}(0,C_0\e)$. First note
  that if $\e$ is small enough compared to $C_0^{-1}$ then Corollary
  \ref{coro corrector} applies and the corrector
  $\p\pth{\wc{\mathbb{U}}}$ exists and satisfy the bounds of Corollary
  \ref{coro corrector}. By construction, this corrector is such that
  \begin{align*}
    \ps{ D\Phi[\bar{g},\bar{\pi}]\pth{(\wc{g},\wc{\pi})+\p\pth{\wc{\mathbb{U}}}} + \RR\pth{\p\pth{\wc{\mathbb{U}}},\wc{\mathbb{U}}} , \mathbb{W}_{j,\ell,\a} }_{L^2(\Sigma)} & = 0 ,
  \end{align*}
  for all $1\leq j \leq q$, $-(j-1)\leq \ell \leq j-1$ and $\a=0,1,2,3$.
  Moreover, using \eqref{DH}, \eqref{DM}, asymptotic flatness, \eqref{assumption check}, \eqref{estim reste 1} and \eqref{estim p(U)} we obtain
  \begin{align*}
    \l D\Phi[\bar{g},\bar{\pi}]\pth{(\wc{g},\wc{\pi})+\p\pth{\wc{\mathbb{U}}}} + \RR\pth{\p\pth{\wc{\mathbb{U}}},\wc{\mathbb{U}}} \r_{L^2_{-q-\de-2}(\Si)} & \lesssim \e (1+C(C_0)\e).
  \end{align*}
  Therefore, the third point in Proposition \ref{prop P final} implies
  that there exists a unique $\Psi^{(s)} \pth{\wc{\mathbb{U}}}$
  solving \eqref{system psi(u,X)} with the bound
  \begin{align*}
    \l\Psi^{(s)} \pth{\wc{\mathbb{U}}} \r_{H^2_{-q-\de}(\Si)} \lesssim \e(1+C(C_0)\e).
  \end{align*}
  Choosing first $C_0$ large enough and then $\e$ small enough compared to
  $C_0^{-1}$ implies that
  $\Psi^{(s)}\pth{\wc{\mathbb{U}}}\in B_{H^2_{-q-\de}(\Si)}(0,C_0\e)$
  and thus that the solution map $\Psi^{(s)}$ is well-defined. 
  
  Next 
  let
  $\wc{\mathbb{U}},\wc{\mathbb{U}}'\in
  B_{H^2_{-q-\de}(\Si)}(0,C_0\e)$. We subtract the two systems
  defining $\Psi^{(s)}\pth{\wc{\mathbb{U}}}$ and
  $\Psi^{(s)}\pth{\wc{\mathbb{U}}'}$ together with Proposition
  \ref{prop P final} and \eqref{estim reste 3} to obtain
  \begin{align*}
    \l \Psi^{(s)}\pth{\wc{\mathbb{U}}} - \Psi^{(s)}\pth{\wc{\mathbb{U}}'} \r_{H^2_{-q-\de}(\Si)} & \lesssim \l D\Phi[\bar{g},\bar{\pi}]\pth{\p\pth{\wc{\mathbb{U}}}-\p\pth{\wc{\mathbb{U}}'} } \r_{L^2_{-q-\de-2}(\Si)} 
    \\&\quad + \l \RR\pth{\p\pth{\wc{\mathbb{U}}},\wc{\mathbb{U}}} - \RR\pth{\p\pth{\wc{\mathbb{U}}'},\wc{\mathbb{U}}'} \r_{L^2_{-q-\de-2}(\Si)}
    \\&\lesssim (1+C(C_0)\e) \l \p\pth{\wc{\mathbb{U}}}-\p\pth{\wc{\mathbb{U}}'} \r_{W^{2,\infty}(\Si)\times W^{1,\infty}(\Si)}
    \\&\quad + C(C_0)\e \l \wc{\mathbb{U}} - \wc{\mathbb{U}}' \r_{H^2_{-q-\de}(\Si)} .
  \end{align*}
  We now use \eqref{estim p(U)-p(U')} to obtain
  \begin{align*}
    \l \Psi^{(s)}\pth{\wc{\mathbb{U}}} - \Psi^{(s)}\pth{\wc{\mathbb{U}}'} \r_{H^2_{-q-\de}(\Si)} & \lesssim  C(C_0)\e \l \wc{\mathbb{U}} - \wc{\mathbb{U}}' \r_{H^2_{-q-\de}(\Si)}.
  \end{align*}
  Taking $\e$ small enough compared to $C_0^{-1}$ shows that $\Psi^{(s)}$ is a contraction.
\end{proof}

In view of Lemma \ref{dernier lemme}, the Banach fixed point theorem implies that if $\e$ is small
enough then there exists a unique solution $\wc{\mathbb{U}}$ to the system
\eqref{eq u X} with the bound
\begin{align*}
  \l \wc{\mathbb{U}} \r_{H^2_{-q-\de}(\Si)} \lesssim \e.
\end{align*}
This concludes the proof of Theorem \ref{theo general}.

\section{Statements and Declarations}

The authors have no competing interests to declare that are relevant to the content of this article. Data sharing not applicable to this article as no datasets were generated or analysed during the current study.

\end{document}